\documentclass[11pt,a4paper]{article}
\title{\bf  Constantin and Iyer's representation formula for the Navier--Stokes equations on manifolds}
\author{Shizan Fang$^1$\footnote{Email: Shizan.Fang@u-bourgogne.fr.}
\quad Dejun Luo$^2$\footnote{Email: luodj@amss.ac.cn.}
\vspace{3mm}\\
{\footnotesize $^1$I.M.B, Universit\'e de Bourgogne, BP 47870, 21078 Dijon, France}\\
{\footnotesize $^2$Key Laboratory of Random Complex Structures and Data Sciences,}\\
{\footnotesize  Academy of Mathematics and Systems Science, Chinese Academy of Sciences, Beijing 100190, China}\\
{\footnotesize $^2$School of Mathematical Sciences, University of the Chinese Academy of Sciences, Beijing 100049, China}
}
\date{}
\usepackage{amssymb,amsmath,amsfonts,amsthm,array,bm,color}

\def\R{\mathbb{R}}
\def\E{\mathbb{E}}

\def\T{\mathbb{T}}

\def\Z{\mathbb{Z}}

\def\I{\mathcal{I}}

\def\G{\mathcal{G}}
\def\K{\mathcal {K}}
\def\M{\mathcal {M}}

\def\L{\mathcal L}
\def\S{\mathbb{S}}

\def\div{\textup{div}}
\def\d{\textup{d}}
\def\Ric{{\rm Ric}}
\def\Ad{{\rm Ad}}
\def\ad{{\rm ad}}
\def\id{{\rm id}}

\def\<{\langle}
\def\>{\rangle}
\let \dis=\displaystyle
\let\ra=\rightarrow

\newtheorem{theorem}{Theorem}[section]
\newtheorem{lemma}[theorem]{Lemma}       
\newtheorem{corollary}[theorem]{Corollary}
\newtheorem{proposition}[theorem]{Proposition}
\newtheorem{remark}[theorem]{Remark}
\newtheorem{example}[theorem]{Example}

\begin{document}

\maketitle
\makeatletter 
\renewcommand\theequation{\thesection.\arabic{equation}}
\@addtoreset{equation}{section}
\makeatother 

\vspace{-8mm}
\begin{abstract}
The purpose of this paper is to establish a probabilistic representation formula for the Navier--Stokes equations on compact Riemannian manifolds. Such a formula has been provided by Constantin and Iyer in the flat case of $\R^n$ or of $\T^n$. On a Riemannian manifold, however, there are several different choices of Laplacian operators acting on vector fields. In this paper, we shall use the de Rham--Hodge Laplacian operator which seems more relevant to the probabilistic setting, and adopt Elworthy--Le Jan--Li's idea to decompose it as a sum of the square of Lie derivatives.
\end{abstract}



\section{Introduction}

The Navier--Stokes equations on a torus $\T^n$ read as
  \begin{equation}\label{NSE}
  \begin{cases}
  \partial_t u+(u\cdot\nabla)u-\nu\Delta u+\nabla p=0,\\
  \nabla\cdot u=0,\quad u|_{t=0}=u_0,
  \end{cases}
  \end{equation}
which describe the evolution of the velocity $u$ of an incompressible viscous fluid with kinematic viscosity $\nu>0$, as well as the pressure $p$. Such equations always attract the attention of many researchers, with an enormous quantity of publications in the literature. Concerning classical results about \eqref{NSE}, we refer to the book \cite{Teman}.  The Lagrangian description of the fluid is to determine the position at time $t$ of the particles of fluid. Due to the high nonlinearity of ordinary differential equations (ODEs for short), such a description was not used too often in the past. However, since the seminal works \cite{DiPernaLions} on the resolution of ODEs with coefficients of low regularity and \cite{Brenier} on the relaxed variational principle for Euler equations, there are more and more interests in Lagrangian descriptions. We refer to \cite{Ambrosio, FangLuo, FangLuoTh, Zhang2, Zhang3} for new developments and various generalizations of \cite{DiPernaLions}, to \cite{Brenier2, AmbrosioF} for generalized flows of Euler equations and to \cite{AC1, AC2, ACF} for generalized stochastic flows of Navier--Stokes equations.

The study of the connections between Navier--Stokes equations and stochastic evolution has a quite long history, which can be traced back to a work of Chorin \cite{Chorin}. Le Jan and Sznitman used in \cite{LeJan} a backward-in-time branching process to express Navier--Stokes equations through Fourier transformations. In \cite{Flandoli}, the authors obtained a representation formula using noisy flow paths for 3-dimensional Navier--Stokes equations. Constantin and Iyer \cite{Constantin} established a probabilistic Lagrangian representation formula by making use of stochastic flows. We also refer to \cite{Constantin} for a more complete description of the history of the developments.

For reader's convenience, let us first state Constantin and Iyer's result \cite[Theorem 2.2]{Constantin}:

\begin{theorem}\label{thm-constantin}
Let $\nu>0$, $W$ be an $n$-dimensional Wiener process, and $u_0\in C^{2,\alpha}$ a given deterministic divergence-free vector field.
Let the pair $(X,u)$ satisfy the stochastic system
  \begin{equation}\label{thm-constantin.1}
  \begin{cases}
  \d X_t=\sqrt{2\nu}\,\d W_t+u_t(X_t)\,\d t,\\
  u_t=\E\mathbf{P}\big[\big(\nabla X_t^{-1}\big)^{\! \top} \big(u_0\circ X_t^{-1}\big)\big],
  \end{cases}
  \end{equation}
where $\mathbf{P}$ is the Leray--Hodge projection and $\top$ denotes the transposition of matrix. Then $u$ satisfies the incompressible Navier--Stokes equations \eqref{NSE}.
\end{theorem}

Based on this stochastic representation, Constantin and Iyer were able to give a self-contained proof of the local existence of the solution to the system \eqref{NSE}. Two proofs of Theorem \ref{thm-constantin} were provided in \cite{Constantin}: the first one relies heavily on the fact that the diffusion coefficient of the stochastic differential equation (SDE) in \eqref{thm-constantin.1} is constant, and transforms it into a random ODE by absorbing the Wiener process into the drift coefficient $u$; the second one applies the generalized It\^o formula  to the quantity $\big(\nabla X_t^{-1}\big)^{\!\top} \big(u_0\circ X_t^{-1}\big)$ which, combined with the stochastic PDE fulfilled by the inverse $X_t^{-1}$, leads to the desired result. Note that if $x\ra u_t(x)$ is $2\pi$-periodic with respect to each component, then SDE \eqref{thm-constantin.1} defines a flow of diffeomorphims on the torus $\T^n$. In order to avoid the computation of the inverse flow $X_t^{-1}$, X. Zhang \cite{Zhang} used the idea that the inverse flow can be described by SDEs driven by time-reversed Brownian motion, and established a similar stochastic representation formula for the backward incompressible Navier--Stokes equations.

The purpose of this note is to extend Constantin and Iyer's representation formula to the Navier--Stokes equations on Riemannian manifolds. To this end, we first give in Section \ref{Alternative-proof} a more geometric interpretation to the formula of $u_t$ in \eqref{thm-constantin.1}, then provide an alternative proof of Theorem \ref{thm-constantin} by making use of Kunita's formula for the pull-back of vector fields under the stochastic flow. Surprisingly enough, it is simpler to work with the inverse flow. More precisely, we get the following expression
  \begin{equation}\label{eq1}
  \int_{\T^n}\<u_t,v\>\,\d x=\E\bigg(\int_{\T^n}\big\<u_0,(X_{t}^{-1})_\ast v\big\>\,\d x\bigg),\quad \forall\, t\geq0,
  \end{equation}
which means that the evolution of $u_t$ in the direction $v$ is equal to the average of the evolution of $v$ under the inverse flow $X_t^{-1}$ in the initial direction $u_0$. The formula \eqref{eq1} has an intrinsic meaning and is suitable to be generalized to Riemannian manifolds.

On a Riemannian manifold $M$, due to the presence of Ricci tensor, there are several ways to define Laplacian operators on vector fields. More precisely, let $\nabla$ be the Levi--Civita connection and $d$ the exterior differential, then we have the covariant Laplacian $\dis\Delta = \hbox{\rm Tr}(\nabla^2)$ and the de Rham--Hodge Laplacian operator $\square=dd^\ast +d^\ast d$. The Weitzenb\"ock formula asserts that
  \begin{equation}\label{Ric1}
  -\square=\Delta - \Ric,
  \end{equation}
where $\Ric$ denotes the Ricci curvature on $M$. In this work, we will be concerned with the opposite $-\square$ of the de Rham--Hodge Laplacian operator, which has a rich literature in stochastic analysis on manifolds, see for example \cite{Elworthy, Malliavin}. Notice that in the geometric setting (cf. \cite{Pierfelice}), the following Laplacian operator
  \begin{equation}\label{Ric2}
  \hat\square=\Delta + \Ric
  \end{equation}
has been used. However, in \cite{TemamW}, Temam and Wang used the de Rham--Hodge operator $\square$.

In Section \ref{Extension}, we shall adopt the idea in \cite{ELL} to decompose $-\square$ as a sum of the square of Lie derivatives on differential forms:
  \begin{equation}\label{Ric3}
  -\square = \sum_{i\in \I} \L_{A_i}^2,
  \end{equation}
where the family $\{A_i:i\in \I\}$ of vector fields might be finite or countable. In general, the vector fields $A_i$ are not of divergence free. See Section \ref{Extension} for the conditions on $\{A_i;\, i\in \I\}$ which ensure such a decomposition. It is surprising that the extra condition
  \begin{equation}\label{Ric4}
  \sum_{i\in \I} \div(A_i) \L_{A_i}B=0\quad\hbox{\rm for any vector field }B
  \end{equation}
is needed so that the decomposition \eqref{Ric3} holds also for vector fields. A new formula in Section \ref{Extension} is
  \begin{equation}\label{Formula2}
  u_t=\E {\mathbf P}\Bigl[\bigl(\rho_t\, (X_t^{-1})^\ast (u_0^\flat) \bigr)^{\sharp}\Bigr]
  \end{equation}
where $\rho_t$ is the Radon--Nikodym density of the associated stochastic flow $X_t$, and we use the musical application $\flat$ (resp. $\sharp$) to transform a vector field $A$ (resp. a differential 1-form $\theta$) to a differential 1-form $A^\flat$ (resp. a vector field $\theta^\sharp$).

The Sections \ref{symmetric} and \ref{isotropic flow} are devoted to examples of vector fields in different spaces which satisfy the conditions (a)--(d) in Section \ref{Extension}. In particular, we give in Section \ref{symmetric} a relatively detailed introduction of the Riemannian symmetric spaces and show that there is a family of Killing vector fields verifying these conditions.  In Section \ref{isotropic flow}, we treat two important examples: tori and spheres, where the divergence-free eigenvector fields of $\square$ enjoy all required properties in Section \ref{Extension}. In all the cases, the vector fields are of divergence free, thus they will generate volume-preserving stochastic flows for which the formula \eqref{Formula2} holds with $\rho_t=1$. Finally, we shall present in Section \ref{appendix} some explicit computations concerning the gradient system $\{A_i;\, 1\leq i\leq n+1\}$ on the sphere $\S^n$, to exhibit the conditions appearing in Section \ref{Extension}.

\section{An alternative proof of Constantin--Iyer's result}\label{Alternative-proof}

Before giving the proof, let us make some preparations. Let $M$ be a compact Riemannian manifold without boundary and $\varphi:M\to M$ a diffeomorphism. Given a vector field $A$ on $M$, the pull-back vector field $(\varphi^{-1})_\ast A$ is defined by
  $$\bigl((\varphi^{-1})_\ast A\bigr) f(x)=A(f\circ \varphi^{-1})(\varphi(x)),\quad\hbox{for any}\quad f\in C^1(M),x\in M.$$
Equivalently,
  \begin{equation}\label{A.1}
  \big((\varphi^{-1})_\ast A\big)(x)=d\varphi^{-1}(\varphi(x))A(\varphi(x))=(d\varphi(x))^{-1}A(\varphi(x)),
  \end{equation}
where $d\varphi$ is the differential of $\varphi$. For two smooth vector fields $A,B$ on $M$, the Lie derivative $\L_A B$ is defined as
  $$(\L_A B)(x)=\lim_{t\to 0}\frac{\big((\varphi^{-1}_{t})_\ast B\big)(x)-B(x)}{t},$$
where $\varphi_t$ is the flow generated by $A$. It is well known that $\L_A B=[A,B]=AB-BA$. We have the following simple result.

\begin{lemma}\label{lem-1}
If $A$ and $B$ are  vector fields of divergence free on $M$, then so is $\L_A B$.
\end{lemma}

\begin{proof}
Since the vector fields $A$ and $B$ are of divergence free, it holds that $\int_M Af\,\d x=\int_M Bf\,\d x=0$ for any function $f\in C^1(M)$. Therefore,
  $$\int_M (\L_A B)f\,\d x=\int_M A(Bf)\,\d x-\int_M B(Af)\,\d x=0,$$
which clearly implies that $\L_A B$ is of divergence free.
\end{proof}

Now we present another proof of Theorem \ref{thm-constantin}, using directly Kunita's formula for the pull-back vector fields under stochastic flows, see \cite[Theorem 2.1, p.265]{Kunita}. Throughout this section, we assume that $u_t$ is in the class $C^{2,\alpha}$ to guarantee that $X_t$ is a stochastic flow of $C^2$-diffeomorphisms.

\begin{proof}[Proof of Theorem \ref{thm-constantin}]
Let $(X,u)$ be the pair solving the system \eqref{thm-constantin.1}. Then $X=(X_t)_{t\geq 0}$ is a stochastic flow of $C^2$-diffeomorphisms on $\T^n$. Since the diffusion coefficient of the SDE is constant and the drift $u$ is of divergence free, we know that the flow $X_t$ preserves the volume measure of the torus $\T^n$. Let $v$ be a vector field of divergence free on $\T^n$, the expression of $u$ in \eqref{thm-constantin.1} gives us
  \begin{align*}
  \int_{\T^n}\<u_t,v\>\,\d x&=\E\bigg(\int_{\T^n}\big\<\big(\nabla X_t^{-1}\big)^{\! \top} \big(u_0\circ X_t^{-1}\big),v\big\>\,\d x\bigg)\\
  &=\E\bigg(\int_{\T^n}\big\<u_0\circ X_t^{-1},\big(\nabla X_t^{-1}\big)v\big\>\,\d x\bigg)\\
  &=\E\bigg(\int_{\T^n}\big\<u_0,\big(\nabla X_t^{-1}(X_t)\big)v(X_t)\big\>\, d x\bigg),
  \end{align*}
where in the last equality we have used the measure-preserving property of $X_t^{-1}$. According to \eqref{A.1}, we get
  \begin{equation}\label{proof.1}
  \int_{\T^n}\<u_t,v\>\,\d x=\E\bigg(\int_{\T^n}\big\<u_0, (X_{t}^{-1})_\ast v\big\>\,\d x\bigg),\quad \forall\, t\geq0.
  \end{equation}

Now by \cite[p.265]{Kunita}, if $u_t$ is of $C^{1,\alpha}$, we have
  $$(X_{t}^{-1})_\ast v= v+\sqrt{2\nu}\sum_{i=1}^n\int_0^t (X_{s}^{-1})_\ast (\partial_i v)\,\d W^i_s
  +\nu \int_0^t (X_{s}^{-1})_\ast (\Delta v)\,\d s+\int_0^t (X_{s}^{-1})_\ast ([u_s,v])\,\d s,$$
where ${\partial_i}v$ denotes the partial derivative of $v$. Substituting this expression of $(X_{t}^{-1})_\ast (v)$ into \eqref{proof.1}, we arrive at
  \begin{equation}\label{proof.2}
  \begin{split}
  \int_{\T^n}\<u_t,v\>\,\d x= &\int_{\T^n}\<u_0,v\>\,\d x+\nu \E\int_0^t\!\!\int_{\T^n}\big\<u_0, (X_{s}^{-1})_\ast (\Delta v)\big\>\,\d x\d s\\
  &+\E\int_0^t\!\!\int_{\T^n}\big\<u_0, (X_{s}^{-1})_\ast ([u_s,v])\big\>\,\d x\d s.
  \end{split}
  \end{equation}
As the vector field $\Delta v$ is of divergence free, we have by \eqref{proof.1} that
  \begin{equation}\label{proof.3}
  \E\int_0^t\!\!\int_{\T^n}\big\<u_0, (X_{s}^{-1})_\ast (\Delta v)\big\>\,\d x\d s=\int_0^t\!\!\int_{\T^n} \<u_s,\Delta v\>\,\d x\d s.
  \end{equation}
Next, Lemma \ref{lem-1} tells us that $[u_s,v]$ is also of divergence free, therefore again by \eqref{proof.1},
  \begin{align*}
  \E\int_0^t\!\!\int_{\T^n}\big\<u_0, (X_{s}^{-1})_\ast ([u_s,v])\big\>\,\d x\d s
  &=\int_0^t\!\!\int_{\T^n} \<u_s,[u_s,v]\>\,\d x\d s\\
  &=\int_0^t\!\!\int_{\T^n} \<u_s, \nabla_{u_s} v- \nabla_v u_s\>\,\d x\d s\\
  &=\int_0^t\!\!\int_{\T^n} \<u_s, \nabla_{u_s} v\>\,\d x\d s-\frac12\int_0^t\!\!\int_{\T^n} v(|u_s|^2)\,\d x\d s\\
  &=\int_0^t\!\!\int_{\T^n} \<u_s, \nabla_{u_s} v\>\,\d x\d s,
  \end{align*}
where in the last equality we have used the fact that $v$ is of divergence free. Substituting this equality and \eqref{proof.3} into \eqref{proof.2}, we obtain for all $t\geq 0$ that
  $$\int_{\T^n}\<u_t,v\>\,\d x=\int_{\T^n}\<u_0,v\>\,\d x+\nu \int_0^t\!\!\int_{\T^n} \<u_s,\Delta v\>\,\d x\d s +\int_0^t\!\!\int_{\T^n} \<u_s, \nabla_{u_s} v\>\,\d x\d s.$$

The above equality implies that for a.e. $t\geq 0$, it holds
  $$\frac{\d}{\d t} \int_{\T^n}\<u_t,v\>\,\d x=\nu \int_{\T^n} \<u_t,\Delta v\>\,\d x+\int_{\T^n} \<u_t, \nabla_{u_t} v\>\,\d x.$$
Multiplying both sides by a real-valued function $\alpha\in C_c^1([0,\infty))$ and integrating by parts, we arrive at
  $$\alpha(0)\int_{\T^n}\<u_0,v\>\,\d x+\int_0^\infty\!\!\int_{\T^n} \big[\alpha'(t)\<u_t,v\>+\nu \alpha(t)\<u_t,\Delta v\>+\alpha(t)\<u_t, \nabla_{u_t} v\>\big]\,\d x\d t=0.$$
Therefore, $u_t$ is a weak solution of the Navier--Stokes equations. Since $u_t$ is assumed to be in the class $C^{2,\alpha}$, it is also a strong solution.
\end{proof}

\section{Navier--Stokes equations on compact Riemannian manifolds}\label{Extension}

In this section, we shall establish the stochastic representation for Navier--Stokes equations on a compact Riemannian manifold $M$ of dimension $n$. To this end, we assume that there exists a (possibly infinite) family of smooth vector fields $\{A_i;\ i\in \I\}$ on $M$ satisfying the following conditions:
\begin{itemize}
\item[\rm(a)] for all $x\in M$, $\dis\sum_{i\in \I}\langle A_i(x), u\rangle_{T_xM}^2=|u|_{T_xM}^2$ for any $u\in T_xM$;
\item[\rm(b)] $\dis  \sum_{i\in \I} \nabla_{A_i} A_i=0$;
\item[\rm(c)] $\dis \sum_{i\in \I}A_i\wedge \nabla_VA_i=0$ for any vector field $V$.
\end{itemize}
Here $\nabla$ denotes the covariant derivative with respect to the Levi--Civita connection on $M$ and $\wedge$ the exterior product. First of all, we give the following example.

\begin{example}[Gradient system]\label{ex1} {\rm
By Nash's embedding theorem, $M$ can be isometrically embedded into $\R^m$ for some $m>n$. For any $x\in M$, denote by $P_x$ the orthogonal projection from $\R^m$ onto $T_xM$. Let $e=\{e_1, \cdots, e_m\}$ be an orthonormal basis of $\R^m$. According to \cite[Section 4.2]{Stroock}, we define
  \begin{equation*}
  A_i(x)=P_x(e_i),\quad x\in M,\, i=1, \cdots, m.
  \end{equation*}
Then $\{A_1, \cdots, A_m\}$ are smooth vector fields satisfying conditions (a), (b) and (c). Note that condition (c) does not often appear. For a justification of (c), we refer to \cite[Remark 2.3.1, p.39]{ELL}. For the case of spheres, we shall do explicit computations in Appendix to illustrate conditions (a), (b) and (c).
}
\end{example}

Now we shall decompose the de Rham--Hodge Laplacian operator $\square$ as the sum of $\L_{A_i}^2$, where $\L_A$ denotes the Lie derivative with respect to the vector field $A$. For a differential form $\omega$ on $M$, it holds that
  \begin{equation}\label{A.31}
  \L_A d\omega=d\L_A\omega,
  \end{equation}
where $d$ is the exterior derivative. Let $I(A)$ be the inner product by $A$, that is, for a differential $q$-form $\omega$,
  \begin{equation*}
  (I(A)\omega)(V_2, \cdots, V_q)=\omega(A, V_2, \cdots, V_q).
  \end{equation*}
Following \cite{ELL}, we define, for a differential $q$-form $\omega$,
  \begin{equation}\label{A.32}
  \hat\delta(\omega)=\sum_{i\in \I} I(A_i)(\L_{A_i}\omega).
  \end{equation}
Let $\delta$ be the divergence operator on differential forms, which admits the expression
  \begin{equation}\label{A.33}
  \delta(\omega)(v_2, \cdots, v_q)=\sum_{j=1}^n (\nabla_{u_j}\omega)(u_j, v_2, \cdots, v_q),
  \end{equation}
where $\{u_1, \cdots, u_n\}$ is an orthonormal basis of $T_xM$.

\begin{proposition}\label{prop3.1}
Under conditions {\rm (a)} and {\rm (b)}, for any differential $1$-form $\omega$, $\hat\delta(\omega)=\delta(\omega)$.
\end{proposition}

\begin{proof}
We have
  \begin{equation}\label{A.34}
  I(A_i)\L_{A_i}\omega=(\L_{A_i}\omega)(A_i)=\L_{A_i}(\omega(A_i))=\omega(\nabla_{A_i}A_i)+(\nabla_{A_i}\omega)(A_i).
  \end{equation}
Let $\{u_1, \cdots, u_n\}$ be an orthonormal basis of $T_xM$, then condition (a) yields
  $$\dis \sum_{i\in \I}\langle A_i(x), u_j\rangle\,\langle A_i(x), u_k\rangle=\langle u_j,u_k\rangle=\delta_{jk}.$$
Therefore, replacing $A_i(x)$ by $\sum_{j=1}^n \langle A_i(x), u_j\rangle u_j$ at the last term in \eqref{A.34}, and summing over $i\in \I$ leads to $\delta(\omega)$ according to \eqref{A.33};  the sum of the first term on the right hand side of \eqref{A.34} vanishes by condition (b).
\end{proof}

\begin{proposition}\label{prop3.2}
Under {\rm (a)}, {\rm (b)} and {\rm (c)},  for any differential $2$-form $\omega$, $\hat\delta(\omega)=\delta(\omega)$.
\end{proposition}

\begin{proof}
By \eqref{A.32}, we have
  \begin{equation*}
  \hat\delta(\omega)(V)=\sum_{i\in \I}(\L_{A_i}\omega)(A_i, V).
  \end{equation*}
Next,
  \begin{equation*}
  \begin{split}
  (\L_{A_i}\omega)(A_i, V)&=\L_{A_i}\bigl(\omega(A_i,V)\bigr)-\omega(A_i, \L_{A_i}V)\\
  &=(\nabla_{A_i}\omega)(A_i,V)+\omega(\nabla_{A_i}A_i,V)+\omega(A_i, \nabla_{A_i}V)-\omega(A_i, \L_{A_i}V)\\
  &=(\nabla_{A_i}\omega)(A_i,V)+\omega(\nabla_{A_i}A_i,V)+\omega(A_i, \nabla_V{A_i}),
  \end{split}
  \end{equation*}
since $\nabla_{A_i}V-\nabla_V A_i=\L_{A_i}V$. By condition (c), $\sum_{i\in \I}\omega(A_i, \nabla_V{A_i})=0.$ Summing over $i\in \I$ and according to (b) and \eqref{A.33}, we get the result.
\end{proof}

Now the opposite of de Rham--Hodge Laplacian operator $-\square =d\delta + \delta d$ admits the following decomposition (see \cite{ELL}):

\begin{theorem}\label{th3.1}
Under the conditions {\rm (a)}--{\rm (c)}, for any differential $1$-form $\omega$, we have
  \begin{equation}\label{A.35}
  \sum_{i\in \I} \L_{A_i}^2\omega=-\square\omega.
  \end{equation}
\end{theorem}

\begin{proof}
Applying Cartan's formula $\dis \L_{A_i}\omega=I(A_i)d\omega+ dI(A_i)\omega$ to $\L_{A_i}\omega$, we have
  \begin{equation*}
  \begin{split}
  \L_{A_i}^2\omega &= I(A_i)d\L_{A_i}\omega+dI(A_i)\L_{A_i}\omega\\
  &=I(A_i)\L_{A_i}(d\omega) + dI(A_i)\L_{A_i}\omega,
  \end{split}
  \end{equation*}
where we used \eqref{A.31} for the second equality. Now by Propositions \ref{prop3.1} and \ref{prop3.2}, we get
  \begin{equation*}
  \sum_{i\in \I} \L_{A_i}^2\omega=\delta d\omega + d\delta\omega=-\square\omega. \qedhere
  \end{equation*}
\end{proof}

Recall that on a Riemannian manifold, there is a one-to-one correspondence between the space of vector fields and that of differential 1-forms. Given a vector field $A$ (resp. differential 1-form $\theta$), we shall denote by $A^\flat$ (resp. $\theta^\sharp$) the corresponding differential 1-form (resp. vector field). The action of the de Rham--Hodge Laplacian $\square$ on the vector field $A$ is defined as follows:
  \begin{equation}\label{Hodge}
  \square A:=(\square A^\flat)^\sharp.
  \end{equation}

\begin{lemma}\label{3-lem}
The conditions {\rm (b)} and {\rm (c)} imply
\begin{equation}\label{A.36}
  \sum_{i\in \I} \div(A_i)A_i=0.
\end{equation}
\end{lemma}

\begin{proof}
We have $I(V)(A_i\wedge \nabla_V A_i)=\langle A_i,V\rangle\nabla_VA_i-\langle\nabla_VA_i,V\rangle A_i$. Let $\{v_1, \cdots, v_n\}$ be an orthonormal basis, then by condition (c),
  \begin{equation*}
  \begin{split}
  0&=\sum_{i\in \I}\sum_{j=1}^n \bigl(\langle A_i,v_j\rangle\nabla_{v_j}A_i-\langle\nabla_{v_j}A_i,v_j\rangle A_i\bigr)\\
  &=\sum_{i\in \I}\nabla_{A_i}A_i - \sum_{i\in \I} \div(A_i)A_i.
  \end{split}
  \end{equation*}
The first term vanishes by condition (b); therefore \eqref{A.36} follows.
\end{proof}

\begin{remark}
When the manifold $M$ is embedded in some $\R^m$, the relation \eqref{A.36} was proved in \cite[p.102]{Stroock}. However, in order to prove the next result, the equality \eqref{A.36} is \emph{not sufficient}; we have to assume the following condition:
\begin{itemize}
\item[\rm(d)] $\dis \sum_{i\in \I} \div(A_i)\L_{A_i}=0$.
\end{itemize}
Unfortunately the vector fields $\{A_1,\cdots, A_m\}$ in Example \ref{ex1} do not satisfy condition {\rm (d)}; see the Appendix.
\end{remark}

\begin{theorem}
Under {\rm (a)}, {\rm (b)}, {\rm (c)} and {\rm (d)}, we have, for any vector field $B$,
  \begin{equation}\label{A.37}
  -\square B=\sum_{i\in \I}\L_{A_i}^2 B.
  \end{equation}
\end{theorem}

\begin{proof}
Let $\omega$ be a differential $1$-form. We have
  $$\L_{A_i} (\omega(B))=(\L_{A_i}\omega)(B)+\omega (\L_{A_i}B),$$
and
  $$\L_{A_i}^2 (\omega(B))=(\L_{A_i}^2\omega)(B)+ \omega(\L_{A_i}^2B)+ 2 (\L_{A_i}\omega)(\L_{A_i}B).$$
By the integration by parts formula,
  \begin{equation*}
  \begin{split}
  \int_M (\L_{A_i}\omega)(\L_{A_i}B)\,\d x&=\int_M \L_{A_i}(\omega(\L_{A_i}B))\,\d x-\int_M \omega(\L_{A_i}^2B)\,\d x\\
  &=-\int_M \div(A_i)\omega(\L_{A_i}B)\,\d x-\int_M \omega(\L_{A_i}^2B)\,\d x.
  \end{split}
  \end{equation*}
Therefore,
  $$\int_M \L_{A_i}^2 (\omega(B))\,\d x=\int_M (\L_{A_i}^2\omega)(B)\,\d x- \int_M \omega(\L_{A_i}^2B)\,\d x -2 \int_M \div(A_i)\omega(\L_{A_i}B)\,\d x.$$
By condition (d), $\sum_{i\in \I}\int_M \div(A_i)\omega(\L_{A_i}B)\,\d x=0$. If we denote by $\tilde\square B=-\sum_{i\in \I} \L_{A_i}^2B$, then summing over $i$ and according to \eqref{A.35}, we get
  $$\int_M \Delta( \omega(B))\,\d x=-\int_M (\square\omega) (B)\,\d x + \int_M \omega (\tilde\square B)\,\d x.$$
It follows that $\int_M (\square\omega) (B)\,\d x= \int_M \omega (\tilde\square B)\,\d x$; therefore $\square B=\tilde\square B$.
\end{proof}

\begin{proposition}\label{prop3.3}
If $\div(B)=0$, then $\div(\square B)=0$.
\end{proposition}

\begin{proof}
Notice first that $\delta(B^\flat)=\div(B)=0$, then by \eqref{Hodge},
  $$\div(\square B)=\delta(\square B^\flat)=\delta d \delta(B^\flat)=0,$$
which completes the proof.
\end{proof}

In what follows, we consider the vector fields $\{A_i;\, i\in \I\}$ which satisfy the conditions (a)--(d). Let $W_t=\{W_t^i; \, i\in\I\}$ be a family of independent standard Brownian motions; consider the Stratonovich SDE on $M$:
  \begin{equation}\label{A.38}
  \d X_t=\sum_{i\in \I} A_i(X_t)\circ \d W_t^i + u_t(X_t)\,\d t,\quad X_0=x\in M.
  \end{equation}
Assume that $u_t\in C^{1,\alpha}$, then $X_t$ is a stochastic flow of $C^1$-diffeomorphisms of $M$. Let
  \begin{equation*}
  \d\big[(X_t)_\#(\d x)\big]=\rho_t\,\d x,\quad \d\big[(X_t^{-1})_\#(\d x)\big]=\tilde\rho_t\, \d x,
  \end{equation*}
where $(X_t)_\#(\d x)$ means the push-forward measure of $\d x$ by $X_t$. By \cite[Lemma 4.3.1]{Kunita2}, $\tilde\rho$ admits the expression
  \begin{equation}\label{A.39}
  \tilde \rho_t(x)=\exp\bigg\{-\sum_{i\in \I} \int_0^t \div(A_i)(X_s(x))\circ \d W_s^i  - \int_0^t \div(u_s)(X_s(x))\,\d s\bigg\}.
  \end{equation}
Since for any $f\in C(M)$, it holds
  \begin{equation*}
  \int_M f(x)\,\d x=\int_M f\big(X_t^{-1}(X_t)\big)\,\d x=\int_M f(X_t^{-1})\rho_t\,\d x=\int_M f\, \rho_t(X_t)\tilde\rho_t\,\d x,
  \end{equation*}
we have
  \begin{equation}\label{A.310}
  \rho_t(X_t)\tilde\rho_t=1.
  \end{equation}

Before stating the main result of this work, we introduce some notations. Let $f: M\ra M$ be a $C^1$-map, then for each $x\in M$, we have the linear operator $df(x): T_x M\to T_{f(x)}M$. We define the adjoint operator $(df)^\ast(x): T_{f(x)}M\to T_xM$ by
  \begin{equation*}
  \langle (df)^\ast (x)v, u\rangle_{T_xM}=\langle df(x)u, v\rangle_{T_{f(x)}M},\quad u\in T_{x}M, v\in T_{f(x)}M.
  \end{equation*}
Let $\omega$ be a differential $1$-form on $M$, the pull-back $f^\ast\omega$ of $\omega$ by $f$ is defined by
  $$\langle f^\ast\omega, u\rangle_x=\langle \omega_{f(x)}, df(x)u \rangle.$$

\begin{theorem}[Stochastic Lagrangian representation]\label{thm-manifold}
Let $M$ be a compact Riemannian manifold such that there is a family of vector fields $\{A_i;\, i\in \I\}$ satisfying the  conditions {\rm (a)}--{\rm (d)}. Let $\nu>0$ and $u_0$ be a divergence-free vector field on $M$. Assume that $u_t\in C^{2,\alpha}$. Then the pair $(X,u)$ satisfies
  \begin{equation}\label{SDE}
  \begin{cases}
  \d X_t=\sqrt{2\nu}\,\sum_{i\in \I} A_i(X_t)\circ \d W^i_t+u_t(X_t)\,\d t,\quad X_0=x,\\
  u_t=\E{\mathbf{P}}\Bigl[\rho_t\, (dX_t^{-1})^\ast\, u_0(X_t^{-1})\Bigr],
   \end{cases}
  \end{equation}
if and only if $u$ solves the Navier--Stokes equations on $M$:
  \begin{equation}\label{NSE-manifold}
  \begin{cases}
  \partial_t u+\nabla_u u+\nu\square u+\nabla p=0,\\
  \div(u)=0,\quad u|_{t=0}=u_0.
  \end{cases}
  \end{equation}
Moreover, $u_t$ has the following more geometric expression
  \begin{equation}\label{A.318}
  u_t=\E\Bigl[ {\mathbf P}\bigl(\rho_t\, (X_t^{-1})^\ast (u_0^\flat)\bigr)^{\sharp}\Bigr].
  \end{equation}
\end{theorem}

\begin{proof}
 Let $v$ be a divergence-free vector field on $M$. We have by \eqref{SDE} that
 \begin{equation*}
 \begin{split}
 \int_M \langle u_t, v\rangle\,\d x&=\E\int_M \rho_t\, \big\< (dX_t^{-1})^\ast\, u_0(X_t^{-1}),\, v\big\>\,\d x\\
 &=\E\int_M \rho_t\, \big\< (dX_t^{-1})\, v, \, u_0(X_t^{-1})\big\> \,\d x\\
 &=\E\int_M \rho_t(X_t)\, \tilde\rho_t\, \big\langle dX_t^{-1}(X_t)v(X_t), \, u_0\big\rangle\,\d x.
 \end{split}
 \end{equation*}
Now using \eqref{A.1} and \eqref{A.310}, we get the following expression, similar to \eqref{proof.1}:
  \begin{equation}\label{A.311}
  \int_M  \langle u_t, v\rangle\,\d x=\E\bigg(\int_M \big\langle u_0, (X_{t}^{-1})_\ast v \big\rangle\,\d x\bigg).
  \end{equation}
Again by \cite[p.265, Theorem 2.1]{Kunita} and \eqref{A.37}, we have
  \begin{align*}
  (X_{t}^{-1})_\ast v
  &=v+\sum_{i\in \I} \int_0^t (X_{s}^{-1})_\ast (\L_{A_i} v)\,\d W^i_s+\nu \sum_{i\in \I} \int_0^t (X_{s}^{-1})_\ast (\L_{A_i}^2 v)\,\d s
  +\! \int_0^t (X_{s}^{-1})_\ast (\L_{u_s}v)\, \d s\\
  &=v+\sum_{i\in \I} \int_0^t (X_{s}^{-1})_\ast(\L_{A_i} v)\,\d W^i_s-\nu  \int_0^t (X_{s}^{-1})_\ast(\square v)\,\d s
  +\int_0^t (X_{t}^{-1})_\ast (\L_{u_s}v)\, \d s.
  \end{align*}
Substituting $(X_{t}^{-1})_\ast v$ into \eqref{A.311}, we have
  \begin{equation*}
  \begin{split}
  \int_M  \langle u_t, v\rangle\,\d x=&\int_M \langle u_0, v\rangle\,\d x -\nu\,  \int_0^t \E\bigg(\int_M \big\langle u_0, (X_{s}^{-1})_\ast(\square v)\big\rangle\,\d x\bigg)\d s\\
  &+\int_0^t \E\bigg(\int_M \big\langle u_0,\, (X_{s}^{-1})_\ast(\L_{u_s}v)\big\rangle\,\d x\bigg) \d s.
  \end{split}
  \end{equation*}

Now by Lemma \ref{lem-1} and Proposition \ref{prop3.3}, $\L_{u_s}v$ and $\square v$
 are of divergence free.  Substituting respectively $v$  in \eqref{A.311} by  $\L_{u_s}v$ and $\square v$ yields
  \begin{equation}\label{thm-manifold.1}
  \int_M\<u_t,v\>\,\d x=\int_M\<u_0,v\>\,\d x-\nu \int_0^t\!\!\int_M \<u_s, \square v \>\,\d x\d s
  +\int_0^t\!\!\int_M \<u_s,\L_{u_s} v \>\,\d x\d s.
  \end{equation}
Since $M$ is torsion-free, we have $\L_{u_s} v=[u_s,v]=\nabla_{u_s} v- \nabla_v u_s$. As a result,
  \begin{equation}\label{thm-manifold.3} \aligned
  \int _M \<u_s,\L_{u_s} v \>\,\d x&=\int_M \<u_s,\nabla_{u_s} v \>\,\d x-\int_M \<u_s,\nabla_v u_s \>\,\d x\cr
  &=\int_M \<u_s,\nabla_{u_s} v \>\,\d x-\frac12 \int_M v(|u_s|^2)\,\d x
  =\int_M \<u_s,\nabla_{u_s} v \>\,\d x.
  \endaligned
  \end{equation}

By \eqref{thm-manifold.1} and \eqref{thm-manifold.3}, we know that for a.e. $t\geq 0$, it holds
  $$\frac{\d}{\d t}\int_M \<u_t,v\>\,\d x=-\nu \int_M \<u_t,\square v \>\,\d x  +\int_M \<u_t,\nabla_{u_t} v \>\,\d x.$$
Multiplying both sides by $\alpha\in C_c^1([0,\infty))$ and integrating by parts on $[0,\infty)$, we arrive at
  \begin{equation*}
  \alpha(0)\int_M\<u_0,v\>\,\d x+\int_0^\infty\!\!\int_M \big[\alpha'(t)\<u_t,v\> +\alpha(t) \<u_t,\nabla_{u_t} v \>
  -\nu\, \alpha(t)\<u_t,\square v\> \big] \d x\d t=0.
  \end{equation*}
The above equation is the weak formulation of the Navier--Stokes \eqref{NSE-manifold} on the manifold $M$. Since $u_t\in C^{2,\alpha}$, it is a strong solution to \eqref{NSE-manifold}.

For proving the converse, we use the idea in \cite[Theorem 2.3]{Zhang}. Let $u_t\in C^{2,\alpha}$ be a solution to \eqref{NSE-manifold}, then for any divergence free vector field $v$,
  \begin{equation}\label{thm-manifold.4}
  \int_M\<u_t,v\>\,\d x=\int_M\<u_0,v\>\,\d x-\nu \int_0^t\!\!\int_M \<u_s, \square v\>\,\d x\d s -\int_0^t\!\!\int_M \<\nabla_{u_s} u_s, v\>\,\d x\d s.
  \end{equation}
Note that
  $$\int_M \<\nabla_{u_s} u_s, v \>\,\d x= \int_M u_s\<u_s, v \>\,\d x -\int_M \<u_s,\nabla_{u_s} v \>\,\d x = -\int_M \<u_s,\nabla_{u_s} v \>\,\d x,$$
which plus \eqref{thm-manifold.3} and \eqref{thm-manifold.4} yields
  \begin{equation*}
  \int_M\<u_t,v\>\,\d x=\int_M\<u_0,v\>\,\d x-\nu \int_0^t\!\!\int_M \<u_s, \square v\>\,\d x\d s  +\int_0^t\!\!\int_M \<u_s,\L_{u_s} v\>\,\d x\d s.
  \end{equation*}
Consider the SDE in  \eqref{SDE} with drift term $u_t$. Define
  \begin{equation*}
  \tilde u_t=\E{\mathbf{P}}\bigl[\rho_t\, (dX_t^{-1})^\ast\, u_0(X_t^{-1})\bigr].
  \end{equation*}
Then the same proof for \eqref{thm-manifold.1} leads to
  \begin{equation*}
  \int_M\<\tilde u_t,v\>\,\d x=\int_M\<u_0,v\>\,\d x-\nu \int_0^t\!\!\int_M \<\tilde u_s, \square v\>\,\d x\d s
  +\int_0^t\!\!\int_M \<\tilde u_s,\L_{u_s} v\>\,\d x\d s.
  \end{equation*}
Let $z_t=u_t-\tilde u_t$; we have
  \begin{equation*}
  \int_M\<z_t,v\>\,\d x=- \nu \int_0^t\!\!\int_M \<z_s, \square v\>\,\d x\d s +\int_0^t\!\!\int_M \<z_s,\L_{u_s} v\>\,\d x\d s.
  \end{equation*}
It follows that $(z_t)$ solves the following heat equation on $M$:
  \begin{equation*}
  \frac{\d z_t}{\d t}=-\nu\,\square z_t - \L_{u_t}^*z_t,\quad z_0=0,
  \end{equation*}
 where $\L_{u_t}^*$ is the adjoint operator.
By uniqueness of solutions, we get that $z_t=0$ for all $t\geq 0$. Thus $u_t=\tilde u_t$.

To prove \eqref{A.318}, we note that
  \begin{align*}
  \int_M \rho_t\, \big\< (X_t^{-1})^\ast (u_0^\flat), v\big\>\,\d x&=\int_M \rho_t\, \big\< u_0^\flat, (X_t^{-1})_\ast v\big\>_{X_t^{-1}}\,\d x\\
  &=\int_M \rho_t(X_t)\,\tilde\rho_t\, \big\< u_0^\flat, (X_t^{-1})_\ast v\big\> \,\d x\\
  &=\int_M \big\< u_0^\flat, (X_t^{-1})_\ast v\big\> \,\d x=\int_M \big\< u_0, (X_{t}^{-1})_\ast v \big\>_{T_x M}\,\d x,
  \end{align*}
where we have used \eqref{A.310} in the third equality. Now by \eqref{A.311}, for any vector field $v$ of divergence free, we have
  $$ \int_M \langle u_t, v\rangle\,\d x=\E\bigg(\int_M \rho_t\, \big\langle (X_t^{-1})^\ast(u_0^\flat), v\big\rangle \,\d x\bigg).$$
Then \eqref{A.318} follows and the proof of Theorem \ref{thm-manifold} is complete.
\end{proof}

\section{Riemannian symmetric spaces}\label{symmetric}

It is usually difficult to find on a general Riemannian manifold a family of vector fields $\{A_i;\, i\in\I\}$ of {\it divergence free}, which satisfy the conditions (a)--(d) in Section \ref{Extension}. In this section, we will treat the case of symmetric spaces.

Let $M$ be a compact Riemannian manifold which is assumed to be symmetric, that is, for each $m\in M$, there is an involutive isometric mapping $s_m$ of $M$ having $m$ as an isolated fixed point. More precisely, $s_m$ is a diffeomorphism of $M$ such that the metric of $M$ is invariant under $s_m$ and $s_m^2$ is the identity map of $M$. Then (see \cite[p.170]{H})
  \begin{equation}\label{S0}
  ds_m(m)=-\id \quad \mbox{on} \ T_mM.
  \end{equation}
Such a map $s_m$ is unique,  sends the geodesic $\gamma(t)$ passing through $m$ to the geodesic $\gamma(-t)$.

Let $G=I_0(M)$ be the identity component of the group of isometries of $M$. Then $G$ has a differential structure to become a Lie group (see \cite[Lemma 3.2, p.171]{H}). Fix a point $m_0\in M$; let $K$ be the subgroup of $G$ such that
  $$K=\{g\in G;\ g(m_0)=m_0\}.$$
Then $K$ is a compact subgroup and the homogeneous space $G/K$ is diffeomorphic to $M$ under the map $[g]=gK \to g(m_0)$. Consider the automorphism $\sigma: G\ra G$ defined by
  \begin{equation}\label{S1}
  g\ra s_{m_0}\circ g\circ s_{m_0}.
  \end{equation}
For simplicity, we denote by $e\in G$ the identity map of $M$. Then $\sigma(e)=e$ and $\sigma^2=\id$ on $G$. Consider the subgroup of fixed point of $\sigma$:
  \begin{equation*}
  K_\sigma=\{g\in G; \sigma(g)=g\}.
  \end{equation*}
For $g\in K$, we see that $\sigma(g)(m_0)=m_0$ and $d\sigma(g)(m_0)=dg(m_0)$ by \eqref{S0} and \eqref{S1}, so that these two isometries $g$ and $\sigma(g)$ coincide (see \cite[Lemma 11.2, p.62]{H}). Thus the following relation holds:
  \begin{equation*}
  K_0\subset K\subset K_\sigma,
  \end{equation*}
where $K_0$ is the identity component of $K_\sigma$.

Let $\G$ be the Lie algebra of $G$; then $d\sigma(e): \G\ra \G$ is an involution. Let
  \begin{equation*}
  \K=\{\xi\in\G;\ d\sigma(e)\,\xi=\xi\},\quad \M=\{\xi\in\G;\ d\sigma(e)\,\xi=-\xi\}.
  \end{equation*}
Then $\G$ is a direct sum of $\K$ and $\M$: $\G=\K\oplus\M$. For any $g\in G$, we denote by ${\rm ad}_g: G\to G$ the inner automorphism and $\Ad_g:\G\to \G$ its differential.

\begin{proposition}\label{propS1} We have
  \begin{equation}\label{S2}
  d\sigma(e)\, [\xi, \eta]=[d\sigma(e)\, \xi, d\sigma(e)\,\eta]\quad \mbox{for any } \xi, \eta\in\G.
  \end{equation}
\end{proposition}

\begin{proof} Let $\exp:\G\to G$ be the exponential map. Consider the map
  $$\Phi(t,s)=\sigma\big(\exp(t\xi)\exp(s\eta)\exp(-t\xi)\big),\quad t, s\in\R.$$
We have $\Phi(t,0)=e$ and
  \begin{equation*}
  \frac{\d}{\d s}\Big|_{s=0}\Phi(t,s)=d\sigma(e)\, \Ad_{\exp(t\xi)}(\eta).
  \end{equation*}
Therefore
  \begin{equation*}
  \frac{\d}{\d t}\Big|_{t=0} \frac{\d}{\d s}\Big|_{s=0} \Phi(t,s)=d\sigma(e)\, [\xi, \eta].
  \end{equation*}
On the other hand,
  $$\Phi(t,s)=\sigma(\exp(t\xi))\,\sigma(\exp(s\eta))\, \sigma(\exp(-t\xi))= \ad_{\sigma(\exp(t\xi))} (\sigma(\exp(s\eta))).$$
The same calculation yields
  \begin{equation*}
  \frac{\d}{\d t}\Big|_{t=0} \frac{\d}{\d s}\Big|_{s=0} \Phi(t,s)=[d\sigma(e)\, \xi, d\sigma(e)\eta].
  \end{equation*}
The relation \eqref{S2} follows.
 \end{proof}

By \eqref{S2}, it is obvious that
  \begin{equation}\label{S3}
  [\K,\K]\subset \K,\ [\M,\M]\subset\K,\ [\K,\M]\subset \M.
  \end{equation}

\begin{proposition}\label{propS2}
We have
\begin{itemize}
\item[\rm(i)]  $\sigma(\exp(t\xi))=\exp(t\xi)$ for $\xi\in\K$,
\item[\rm(ii)] $\sigma(\exp(t\xi))=\exp(-t\xi)$ for $\xi\in\M$.
\end{itemize}
\end{proposition}

\begin{proof} Let $\xi\in\K$ and $\varphi(t)= \sigma(\exp(t\xi))$; then $\varphi(t)\varphi(s)=\varphi(t+s)$ for $t,s\in\R$. Hence $\{\varphi(t); t\in\R\}$ is a one-parameter subgroup of $G$ such that $\varphi(0)=e$ and $\varphi'(0)=d\sigma(e)\,\xi=\xi$. Therefore $\varphi(t)=\exp(t\xi)$ and we get (i). The same proof also works for (ii).
\end{proof}

As a corollary of this result, $\M$ is invariant under $\Ad_K$. In fact, for any $h\in K,\, \xi\in\M$, we have
  \begin{equation*}
  \sigma\big( \exp(t \Ad_h(\xi))\big)=\sigma\big( h \exp(t\xi) h^{-1}\big)=\sigma(h) \sigma (\exp(t\xi))\sigma(h)^{-1}=h \exp(-t\xi) h^{-1}.
  \end{equation*}
Taking the derivative with respect to $t$ at $t=0$, we get $\dis d\sigma(e)\, \Ad_h(\xi)=-\Ad_h(\xi)$. Therefore $\Ad_h(\xi)\in\M$. Similarly, we can show that $\K$ is $\Ad_K$-invariant.

For any $\xi\in \K$, the assertion (i) implies that $\exp(t\xi)\in K_\sigma$ for all $t\in\R$. Then $s_{m_0}\exp(t\xi)(m_0)= \exp(t\xi)(m_0)$ for all $t\in\R$. Since $m_0$ is the isolated fixed point of $s_{m_0}$, we have $\exp(t\xi)(m_0)=m_0$ for $t\in \R$. We see in fact that $\exp(t\xi)\in K_0\subset K$ and $\K$ is the Lie algebra of $K$.
Now we consider the map $\pi: G\ra M$ defined by $\pi(g)=g(m_0)$. Then
$$\dis d\pi(e): \G \ra T_{m_0}M.$$
For $\xi\in\M$, the curve $\gamma(t)=\exp(t\xi)(m_0)$ is the geodesic on $M$ starting from $m_0$ such that $\gamma'(0)=d\pi(e)\,\xi$. Moreover, $\K= \textup{Ker}(d\pi(e))$ and $d\pi(e): \M\ra T_{m_0}M$ is an isomorphism (see \cite[p.173]{H}).

Now for $\xi\in \G$, we define
  \begin{equation}\label{S4}
  A_\xi(m)= \frac{\d}{\d t}\Big|_{t=0} \exp(t\xi)(m),\quad m\in M.
  \end{equation}
The vector field $A_\xi$ is a Killing vector field on $M$; in fact $\exp(t\xi): M\ra M$ is an isometry which leaves the metric of $M$ invariant. Let $\d g$ be the Haar measure on $G$ and $\d m=\pi_\# \d g$. Then for any $\xi\in\G$ and $f\in C^1(M)$
  $$\int_M f(\exp(t\xi)(m))\, \d m=\int_G f(\exp(t\xi)\, g(m_0))\, \d g=\int_G f(g(m_0))\, \d g.$$
Taking the derivative with respect to $t$, at $t=0$, we get
  \begin{equation}\label{S5}
  \int_M A_\xi f (m)\,\d m=0, \quad \textup{for } f\in C^1(M).
  \end{equation}
In other words, $\div (A_\xi)=0$. If we denote $R_m(g)=g(m)$ for $m\in M$, then $A_\xi(m)= dR_m(e)\, \xi$. The dependence $\xi\ra A_\xi$ is linear from $\G$ to $\mathcal X(M)$, where $\mathcal X(M)$ is the space of vector fields on $M$.

\begin{proposition}\label{propS3} We have for $\xi,\eta\in\G$,
\begin{equation}\label{S6}
A_{[\xi,\eta]}=-[A_\xi, A_\eta].
\end{equation}
\end{proposition}

\begin{proof} Consider $\Psi(t,s)=\exp(t\xi) \exp(s\eta) \exp(-t\xi)(m)$ for $m\in M$ and $t, s\in\R$. We have $ \Psi(t,s)=\exp\big({s\Ad_{\exp(t\xi)}(\eta})\big)(m)$. Thus,
  \begin{equation*}
  \frac{\d}{\d s}\Big|_{s=0} \Psi(t,s)=A_{\Ad_{\exp(t\xi)}(\eta)}(m),\quad
  \frac{\d}{\d t}\Big|_{t=0} \frac{\d}{\d s}\Big|_{s=0} \Psi(t,s)=A_{[\xi,\eta]}(m).
  \end{equation*}
On the other hand,
  \begin{equation*}
  \frac{\d}{\d s}\Big|_{s=0} \Psi(t,s)=d\exp(t\xi)\big(\exp(-t\xi)(m)\big)\, A_\eta\big(\exp(-t\xi)(m)\big)=\big((\exp(t\xi))_\ast A_\eta\big) (m).
  \end{equation*}
Therefore
  \begin{equation*}
  \frac{\d}{\d t}\Big|_{t=0} \frac{\d}{\d s}\Big|_{s=0} \Psi(t,s)= \big({\mathcal L}_{A_{-\xi}}A_\eta\big)(m)=-[A_\xi, A_\eta](m).
  \end{equation*}
The result follows.
\end{proof}

\begin{proposition}\label{propS4} We have, for any $\xi\in\G$, $g\in G$ and $m\in M$,
  \begin{equation}\label{S7}
  A_\xi(g(m))=dg(m)\, A_{\Ad_{g^{-1}}(\xi)}(m),\quad\textup{or}\quad (g^{-1})_*A_\xi= A_{\Ad_{g^{-1}}(\xi)}.
  \end{equation}
\end{proposition}

\begin{proof} The first relation in \eqref{S7} follows by taking derivative at $t=0$ of the equality below:
  \begin{equation*}
  \exp(t\xi)(g(m))=\big(g\circ g^{-1}\circ \exp(t\xi)\circ g\big)(m)= g\bigl( \exp\big(t\Ad_{g^{-1}}(\xi)\big)(m)\bigr).
  \end{equation*}
The second one deduces from the first one.
\end{proof}

Now we need an inner product on $\G$ which is $\Ad_G$-invariant such that $\K\perp\M$. The Killing form $B$ will play such role. For $\xi\in\G$, we denote by $\ad(\xi)(\eta)=[\xi, \eta]$ which defines a linear map from $\G$ to $\G$. The Killing form is defined by
  $$ B(\xi, \eta)={\rm Tr}( \ad(\xi)\circ\ad(\eta) ),\quad \xi,\eta\in\G.$$
Using \eqref{S2}, we have
  $$ \ad(d\sigma(e)\, \xi)\circ \ad(d\sigma(e)\, \eta)=d\sigma(e)\circ \ad(\xi)\circ \ad(\eta)\circ d\sigma(e)^{-1},$$
which implies that $B(d\sigma(e)\xi, d\sigma(e)\eta)=B(\xi,\eta)$. Therefore
  $$B(\xi,\eta)=0\quad\textup{if  } \xi\in\K,\, \eta\in\M.$$

In the sequel we assume that $-B$ is positive definite on $\G\times\G$, which is the case if $G$ is compact and semi-simple. In what follows, we will denote by
  $$\langle\xi,\eta\rangle_\G=-B(\xi,\eta),\quad \xi,\eta\in\G.$$
We shall transport the metric $\langle\ ,\, \rangle_\G$ on $\G$ to $T_{m_0}M$ by $d\pi(e)$. Define
  \begin{equation}\label{S8}
  \langle A_\xi, A_\eta\rangle_{m_0}=\langle \xi,\eta\rangle_\G\quad\textup{for } \xi, \eta\in \M.
  \end{equation}
Equivalently,
  \begin{equation}\label{S9}
  \langle A_\xi, A_\eta\rangle_{m_0}=\langle P_\M(\xi), P_\M(\eta)\rangle_\G\quad\textup{for } \xi, \eta\in \G,
  \end{equation}
where $P_\M$ is the projection from $\G$ onto $\M$. Note that if $h\in K$, $dh(m_0)$ is an isometric transform of $ T_{m_0}M$. According to \eqref{S7}, for $\xi, \eta\in\M$,
  \begin{equation*}
  \begin{split}
  &\big\langle dh(m_0)A_\xi(m_0), dh(m_0)A_\eta(m_0)\big\rangle_{m_0}
  =\big\langle A_{\Ad_h(\xi)},A_{\Ad_h(\eta)}\big\rangle_{m_0}\\
  &=\langle \Ad_h(\xi), \Ad_h(\eta)\rangle_{\G} =\langle \xi, \eta\rangle_{\G} =\langle A_\xi, A_\eta\rangle_{m_0}.
  \end{split}
  \end{equation*}
Therefore $\langle\  ,\,  \rangle_{m_0}$ will define a Riemannian metric on $M$ which is $G$-invariant.

Now let $m\in M$ with $m=g(m_0)$. For any $u\in T_m M$, there is a unique $v\in T_{m_0}M$ such that $u=dg(m_0)v$. Furthermore, me can take $\xi_0\in \M$ such that $v= A_{\xi_0}(m_0)$. We have $|u|_m =|v|_{m_0} =|\xi_0|_{\G}$. For any $\xi\in \G$, by \eqref{S7},
  $$\<A_\xi(m), u\>_m =\<A_\xi(g(m_0)), u\>_m =\big\<dg(m_0) A_{\Ad_{g^{-1}}(\xi)}(m_0), dg(m_0) A_{\xi_0}(m_0)\big\>_m.$$
As $dg(m_0): T_{m_0} M\to T_{m} M$ is an isometry, we have
  $$\<A_\xi(m), u\>_m =\big\<A_{\Ad_{g^{-1}}(\xi)}(m_0), A_{\xi_0}(m_0)\big\>_{m_0} = \big\<\Ad_{g^{-1}}(\xi), \xi_0\big\>_{\G}.$$
Let $\{\xi_1, \cdots, \xi_n\}$ be an orthonormal basis of $\G$. Since $\langle\ ,\, \rangle_\G$ is assumed to be $\Ad_G$-invariant, $\big\{\Ad_{g^{-1}}(\xi_1),\cdots, \Ad_{g^{-1}}(\xi_n) \big\}$ is again an orthonormal basis of $\G$. Therefore,
  \begin{equation}\label{S10}
  \sum_{i=1}^n \langle A_{\xi_i}, u\rangle_m^2 =\sum_{i=1}^n \big\<\Ad_{g^{-1}}(\xi_i), \xi_0\big\>_{\G}^2 =|\xi_0|_\G^2= |u|_m^2, \quad\textup{for  } u\in T_mM.
  \end{equation}
Thus we see that the Killing vector fields $\{A_{\xi_1}, \cdots, A_{\xi_n}\}$ satisfy the condition (a). To verify the conditions (b) and (c), we need some more preparations.

\begin{proposition}\label{propS5} Let $\nabla$ be the associated Levi--Civita connection on $M$, then at $m_0$, we have
for $\xi, \eta, \zeta\in\G$ that
  \begin{equation}\label{S11}
  \big\<\nabla_{A_\xi} A_\eta, A_\zeta\big\> = \frac12 \big(\<[\zeta, \xi], P_{\M}\eta\> - \<[\xi, \eta], P_{\M} \zeta\> - \<[\eta, \zeta], P_{\M}\xi\> \big).
  \end{equation}
\end{proposition}

\begin{proof}
We first show that for any Killing vector fields $X, Y$ and $Z$ on $M$, it holds
  $$\<\nabla_X Y,Z\>=\frac12 \big(\<[X,Y],Z\>+ \<[Y,Z],X\> -\<[Z,X],Y\>\big).$$
Since $X$ is a Killing vector field, we have
  $$\<\nabla_Y X, Z\>+\<\nabla_Z X, Y\>=0.$$
Combining this identity with $[X,Y]=\nabla_X Y- \nabla_Y X$ yields that
  $$\<\nabla_X Y, Z\> + \<\nabla_Z X, Y\>=\<[X,Y],Z\>.$$
As $Y$ and $Z$ are also Killing vector fields, we obtain in the same way that
  $$\<\nabla_Y Z, X\> + \<\nabla_X Y, Z\>=\<[Y,Z],X\>$$
and
  $$\<\nabla_Z X, Y\> + \<\nabla_Y Z, X\>=\<[Z,X],Y\>.$$
Adding the first two equalities and subtracting the third one give us the desired result.

Now for any $\xi,\eta,\zeta\in\G$, applying the above result leads to
  $$\big\<\nabla_{A_\xi} A_\eta, A_\zeta\big\>=\frac12 \big(\<[A_\xi, A_\eta], A_\zeta\>+ \<[A_\eta, A_\zeta], A_\xi\> -\<[A_\zeta, A_\xi], A_\eta\>\big).$$
According to \eqref{S6}, this equality can be rewritten as
  $$\aligned \big\<\nabla_{A_\xi} A_\eta, A_\zeta\big\>&= \frac12 \big(-\<A_{[\xi, \eta]}, A_\zeta\> - \<A_{[\eta, \zeta]}, A_\xi\> +\<A_{[\zeta, \xi]}, A_\eta\>\big)\\
  & =\frac12 \big(\<[\zeta, \xi], P_{\M}\eta\> - \<[\xi, \eta], P_{\M} \zeta\> - \<[\eta, \zeta], P_{\M}\xi\> \big),\endaligned$$
where in the second step we have used \eqref{S9}.
\end{proof}

\begin{corollary}\label{RSS-cor}
For any $\xi,\eta\in\M$, it holds
  \begin{equation}\label{S11.1}
  \nabla_{A_\xi} A_\eta(m_0)=0.
  \end{equation}
\end{corollary}

\begin{proof}
For any $v\in T_{m_0} M$, there is $\zeta\in \M$ such that $v= A_{\zeta}(m_0)$. By \eqref{S11} and \eqref{S3}, it is clear that $\big\<\nabla_{A_\xi} A_\eta, v\big\>_{m_0} =0$. The arbitrariness of $v\in T_{m_0} M$ implies the desired result.
\end{proof}

From now on, we assume that $\{\xi_1, \cdots,\xi_d\}$  is an orthonormal basis of $\M$ and $\{\xi_{d+1}, \cdots,\xi_n\}$  is an orthonormal basis of $\K$, then by \eqref{S11.1},
  \begin{equation}\label{S12}
  \sum_{i=1}^n \nabla_{A_{\xi_i}}A_{\xi_i}(m_0)=0,
  \end{equation}
since $A_{\xi_i}(m_0)=0$ for $i\in \{d+1,\cdots, n\}$.

In order to transfer the above property from the base point $m_0$ to any point $m\in M$, we use the fact that the affine connection enjoys the following relation (see \cite[Chap. 1]{H}): for any vector field $X$ on $M$,
  \begin{equation}\label{S13}
  \bigl[\nabla_{dg(m_0)v}(g_\ast X)\bigr](g(m_0))=dg(m_0) (\nabla_v X)(m_0),\quad v\in T_{m_0}M.
  \end{equation}
Therefore  replacing $X$ in \eqref{S13} by $(g^{-1})_\ast A_\xi$, we get
  \begin{equation}\label{S14}
  \bigl(\nabla_{dg(m_0)v}A_\xi\bigr) (g(m_0))=dg(m_0)\bigl[\nabla_v \bigl((g^{-1})_\ast A_\xi\bigr)\bigr](m_0).
  \end{equation}
Let $m\in M$ with $m=g(m_0)$. By \eqref{S7},
  \begin{equation*}
  \bigl(\nabla_{A_\xi}A_\xi\bigr)(m)=\bigl(\nabla_{dg(m_0)v}A_\xi\bigr)(g(m_0)),\quad\textup{where  } v=A_{\Ad_{g^{-1}} (\xi)}(m_0).
  \end{equation*}
Again by the second formula in \eqref{S7} and \eqref{S14}, we get
  \begin{equation}\label{S14.1}
  \bigl(\nabla_{A_\xi}A_\xi\bigr)(m)=dg(m_0)\Bigl[\nabla_{A_{\Ad_{g^{-1}}(\xi)}} A_{\Ad_{g^{-1}}(\xi)}\Bigr](m_0).
  \end{equation}

Recall that $\big\{\Ad_{g^{-1}}(\xi_1),\cdots, \Ad_{g^{-1}}(\xi_n) \big\}$ is also an orthonormal basis of $\G$, hence there is an orthogonal matrix $U=(u_{ij})$ of order $n$ such that
  \begin{equation}\label{S14.5}
  \Ad_{g^{-1}}(\xi_i)=\sum_{j=1}^n u_{ij} \xi_j,\quad i=1,\cdots, n.
  \end{equation}
Combining \eqref{S14.1} and \eqref{S14.5} yields that
  $$\aligned \sum_{i=1}^n \big(\nabla_{A_{\xi_i}}A_{\xi_i}\big)(m)&= \sum_{i=1}^n\sum_{j,k=1}^n u_{ij} u_{ik}\, dg(m_0)\bigl(\nabla_{A_{\xi_j}} A_{\xi_k}\bigr)(m_0)\\
  &= \sum_{j=1}^n dg(m_0)\bigl(\nabla_{A_{\xi_j}} A_{\xi_j}\bigr)(m_0)=0, \endaligned$$
where the last equality follows from \eqref{S12}.

It remains to check condition (c) in Section \ref{Extension}.  By Corollary \ref{RSS-cor}, it is clear that for any $v\in T_{m_0} M$ and $\xi\in\M$, we have $\nabla_v A_\xi(m_0)=0$. Therefore, by the choice of $\{\xi_1,\cdots,\xi_n\}$,
  \begin{equation}\label{S15}
  \sum_{i=1}^n \langle A_{\xi_i}, v_1\rangle_{m_0}\langle \nabla_{v_3}A_{\xi_i}, v_2\rangle_{m_0}=0,   \quad\textup{for any  } v_1, v_2, v_3\in T_{m_0}M.
  \end{equation}
Now for $m=g(m_0)$, $u_j=dg(m_0)v_j\in T_m M, \, j=1, 2, 3$. Applying \eqref{S7} and \eqref{S14}, we get
  \begin{equation*}\aligned
  \langle A_\xi, u_1\rangle_m \big\langle \nabla_{u_3}A_\xi, u_2\big\rangle_m
  &=\big\langle A_{\Ad_{g^{-1}}(\xi)}(m_0), v_1\big\rangle_{m_0} \big\langle \nabla_{v_3}\big((g^{-1})_\ast A_\xi\big)(m_0), v_2\big\rangle_{m_0}.
  \endaligned\end{equation*}
Using the second assertion of  \eqref{S7}, we arrive at
  $$\langle A_\xi, u_1\rangle_m \big\langle \nabla_{u_3}A_\xi, u_2\big\rangle_m= \big\langle A_{\Ad_{g^{-1}} (\xi)}(m_0),v_1\big\rangle_{m_0} \big\langle \nabla_{v_3} A_{\Ad_{g^{-1}} (\xi)}(m_0),  v_2\big\rangle_{m_0}.$$
Therefore, applying this equality to $\xi=\xi_i$ and by \eqref{S14.5}, \eqref{S15}, we finally get
  \begin{equation}\label{S16}
  \sum_{i=1}^n \langle A_{\xi_i}, u_1\rangle_m \langle \nabla_{u_3}A_{\xi_i}, u_2\rangle_m = 0, \quad \mbox{for any } u_1, u_2, u_3\in T_mM.
  \end{equation}
This immediately implies the condition (c). Summing up the above discussions, we have proved

\begin{theorem} Let $M$ be a compact symmetric Riemannian manifold and $\G$ the Lie algebra of the group of isometries of $M$. Assume that the minus Killing form $-B$ on $\G$ defines an inner product, and the orthonormal basis  $\{\xi_1, \ldots, \xi_n\}$ of $\G$ fulfils $\M={\rm span}\{\xi_1, \cdots, \xi_d\}$ and $\K={\rm span}\{\xi_{d+1}, \cdots, \xi_n\}$. Then the family of vector fields $\{A_{\xi_1}, \cdots, A_{\xi_n}\}$ enjoy properties {\rm (a)}--{\rm (d)} in Section \ref{Extension} for the metric induced by $-B$.
\end{theorem}

The following explicit example of the unit sphere is taken from \cite[Chap. 9, Example 4.2]{ChenWH}.

\begin{example} {\rm
Recall that the special orthogonal group ${\rm SO}(n+1)$ consists of orthogonal matrix of order $n+1$ whose determinant is 1. It is a connected compact Lie group. Let
  $$s=\begin{pmatrix} I_n &\ 0\\
  0 &\ -1
  \end{pmatrix},$$
where $I_n$ is the identity matrix of order $n$. Then $s^2=I_{n+1}$, that is $s^{-1}=s$. Define $\sigma:{\rm SO}(n+1)\to {\rm SO}(n+1)$ as follows:
  $$\sigma(U)=s\, Us, \quad U\in {\rm SO}(n+1).$$
$\sigma$ is an involution on ${\rm SO}(n+1)$, i.e. $\sigma^2={\rm id.}$ Assume that $U\in {\rm SO}(n+1)$ satisfies $\sigma(U)=U$, that is $s\,U=Us$, then $U$ must have the form
  $$U=\begin{pmatrix} V &\ 0\\
  0 &\ \det V
  \end{pmatrix}, \quad V\in {\rm O}(n),$$
where $\det V$ is the determinant of $V$ and ${\rm O}(n)$ is the orthogonal group of order $n$. Therefore, the subgroup of ${\rm SO}(n+1)$ consists of the fixed points of $\sigma$ is
  $$K_\sigma=\left\{\begin{pmatrix} V &\ 0\\
  0 &\ \det V
  \end{pmatrix};\, V\in {\rm O}(n)\right\} \cong {\rm O}(n),$$
which is also a closed subgroup of ${\rm SO}(n+1)$, hence a compact subgroup. The identity component of $K_\sigma$ is
  $$K_0=\left\{\begin{pmatrix} V &\ 0\\
  0 &\ 1
  \end{pmatrix};\, V\in {\rm SO}(n)\right\} \cong {\rm SO}(n).$$

The Lie algebra of ${\rm SO}(n+1)$ is
  $${\rm so}(n+1)=\left\{\begin{pmatrix} X &\ a\\
  -a^{ \top} &\ 0
  \end{pmatrix};\, a\in\R^n, X^{\top}=-X \right\}, $$
where $a^{\top}$ is the transposition of $a\in\R^n$, and that of $K_0$ is
  $$\K=\left\{\begin{pmatrix} X &\ 0\\
  0 &\ 0
  \end{pmatrix};\, X^{\top}=-X \right\}. $$
The involution on ${\rm so}(n+1)$ induced by $\sigma$ is
  $$d\sigma(I_{n+1}) (\tilde X)=s\tilde X s,\quad \tilde X\in {\rm so}(n+1).$$
Hence
  $$\mathcal M=\big\{\tilde X\in {\rm so}(n+1): d\sigma(I_{n+1}) (\tilde X)=-\tilde X\big\}= \left\{\begin{pmatrix} 0 &\ a\\
  -a^{\top} &\ 0
  \end{pmatrix};\, a\in\R^n\right\} \cong \R^n.$$
It is known that the Killing form on ${\rm so}(n+1)$ is given by (see \cite[p.266]{KN})
  \begin{equation}\label{sphere-1}
  B\big(\tilde X,\tilde Y\big)=(n-1){\rm Tr}\big(\tilde X \tilde Y\big),\quad \tilde X, \tilde Y\in {\rm so}(n+1),
  \end{equation}
which is ${\rm Ad}_{{\rm SO}(n+1)}$-invariant.

We explain now the geometric meaning of ${\rm SO}(n+1)/{\rm SO}(n)$. Let $U\in {\rm SO}(n+1)$. Then the column vectors $u_1,\cdots, u_{n+1}$ of  $U$ constitute an orthonormal basis of $\R^{n+1}$. The left coset $[U]=U\cdot K_0$ is a collection of orthonormal bases of $\R^{n+1}$:
  $$[U]=\big\{(\tilde u_1,\cdots, \tilde u_{n+1})\in {\rm SO}(n+1): \tilde u_{n+1}=u_{n+1}\big\} = \left\{U\cdot \begin{pmatrix} V &\ 0\\
  0 &\ 1
  \end{pmatrix};\, V\in {\rm SO}(n)\right\}.$$
Therefore, $[U]$ consists of those orthonormal basis $\{\tilde u_1,\cdots, \tilde u_{n+1}\}$ of $\R^{n+1}$ such that $\tilde u_{n+1}=u_{n+1}$  is fixed and they have the same orientation with $\{u_1,\cdots, u_{n+1}\}$. We define the map $\varphi: {\rm SO}(n+1)/{\rm SO}(n)\to S^n\subset \R^{n+1}$ such that
  \begin{equation}\label{sphere-0}
  \varphi([U])= u_{n+1},
  \end{equation}
which is a smooth diffeomorphism.

Next we consider the Riemannian metric on ${\rm SO}(n+1)/{\rm SO}(n)$. For any $a,b\in \R^n$, let
  $$\tilde X=\begin{pmatrix} 0 &\ a\\ -a^{\top} &\ 0\end{pmatrix},\quad \tilde Y=\begin{pmatrix} 0 &\ b\\ -b^{\top} &\ 0\end{pmatrix} \in \mathcal M.$$
Then
  $$\tilde X \tilde Y=\begin{pmatrix} -ab^{\top} &\ 0 \\ 0 & \ -a^{\top} b \end{pmatrix}.$$
Consequently, by \eqref{sphere-1},
  \begin{equation*}
  -\frac1{2(n-1)}B\big(\tilde X,\tilde Y\big) =-\frac12 {\rm Tr}\big(\tilde X \tilde Y\big)= a^{\top}\, b =\<a,b\>,
  \end{equation*}
where $\<\ , \, \>$ is the inner product in $\R^n$. Thus, $-B/2(n-1)$ induces an ${\rm SO}(n+1)$-invariant Riemannian metric on ${\rm SO}(n+1)/{\rm SO}(n)$, such that $\varphi$ defined in \eqref{sphere-0} is an isometry.

Finally we define the fundamental vector fields on $S^n$. For $V\in {\rm SO}(n+1)$, the action $\tau(V)$ of $V$ on ${\rm SO}(n+1)/{\rm SO}(n)$ is
  $$\tau(V)([U]) = [VU]= (VU) K_0.$$
Denote by $\tilde\tau(V)$ the action of $V$ on $S^n$ induced by $\varphi$, that is $\tilde\tau(V)= \varphi\circ \tau(V) \circ \varphi^{-1}$. Thus for any $u \in S^n$,
  $$\tilde\tau(V)(u) = V u.$$
Then for any $\tilde X\in {\rm so}(n+1)$,
  $$A_{\tilde X}(u)=\frac{\d}{\d t}\Big|_{t=0} \tilde\tau \big(\exp\big(t\tilde X\big)\big)(u)= \frac{\d}{\d t}\Big|_{t=0} \exp\big(t\tilde X\big)\,u =\tilde X u,\quad \mbox{for all } u\in S^n\subset \R^{n+1}.$$
Fix any pair $(i,j)$ of integer index with $1\leq i <j\leq n+1$, let $\tilde X^{(ij)}\in {\rm so}(n+1)$ be such that for all $1\leq k<l\leq n+1$,
  $$\tilde X^{(ij)}_{kl}=\begin{cases} 1, & \mbox{if } k=i,\, l=j;\\
  0, & \mbox{otherwise}.
  \end{cases}$$
Then the family of fundamental vector fields $\big\{A_{\tilde X^{(ij)}}: 1\leq i <j\leq n+1\big\}$ verify our requirements.

}
\end{example}

\section{Volume-preserving flows on the torus and the sphere}\label{isotropic flow}

The group $\textup{Diff}(M)$ of diffeomorphisms of $M$ plays an important role in the description of fluid mechanics.
 In this part, we shall treat two important examples: torus $\T^n$ and sphere $\S^n$.

\subsection{Case of torus $\T^n$}

Let $\Z^n$ be the set of lattice points in $\R^n$ and define $\Z^n_0=\Z^n\setminus \{0 \}$, where $0$ means the zero vector in $\R^n$. For $x,y$ in $\R^n$, we denote by $x\cdot y$ or $\<x,y\>$ the scalar product. For $k\in \Z_0^n$, we denote by $k^\perp$ the $(n-1)$-dimensional subspace of $\R^n$ which is orthogonal to $\{k\}$, and we fix an orthonormal basis $\{e_{k,1}, \cdots, e_{k,n-1}\}$ of $k^\perp$. In the two dimensional case, we have the explicit choice $e_{k,1}=(k_2,-k_1)/|k|$. We fix some constant $\beta>n/2$ and define
  $$A_{k,i}(\theta)=\frac{\cos(k\cdot \theta)}{|k|^\beta}e_{k,i},\quad B_{k,i}(\theta)=\frac{\sin(k\cdot \theta)}{|k|^\beta}e_{k,i},\quad \theta\in\T^n,\ 1\leq i\leq n-1.$$
Since $\<k,e_{k,i}\>=0$, it is clear that these vector fields are of divergence free. Moreover, the family $\{A_{k,i},B_{k,i}:1\leq i\leq n-1, k\in \Z_0^n\}$ is a complete orthogonal system of the space of divergence free vector fields $V$ on $\T^n$ such that $\int_{\T^n} V\,\d\theta=0$. We shall check in the following  that they satisfy the conditions (a), (b) and (c).

First, for any $u\in \R^n$,
  $$\<A_{k,i}(\theta),u\>^2+\<B_{k,i}(\theta),u\>^2=\frac{\cos^2(k\cdot\theta)+\sin^2(k\cdot\theta)}{|k|^{2\beta}}\<e_{k,i},u\>^2 =\frac{\<e_{k,i},u\>^2}{|k|^{2\beta}}.$$
Hence
  \begin{equation}\label{torus-n.1}
  \sum_{i=1}^{n-1}\big(\<A_{k,i}(\theta),u\>^2+\<B_{k,i}(\theta),u\>^2\big)
  =\sum_{i=1}^{n-1}\frac{\<e_{k,i},u\>^2}{|k|^{2\beta}}
  =\frac{1}{|k|^{2\beta}}\bigg(|u|^2-\frac{\<u,k\>^2}{|k|^2}\bigg).
  \end{equation}
We have
  \begin{equation}\label{torus-n.2}
  \<u,k\>^2=\sum_{i=1}^n u_i^2k_i^2+\sum_{1\leq i\neq j\leq n} u_iu_j k_ik_j.
  \end{equation}

\begin{lemma}\label{sect-4-lem}
For any $i,j\in \{1,\cdots, n\}$ with $i\neq j$,
  $$\sum_{k\in \Z_0^n} \frac{k_ik_j}{|k|^{2\beta+2}}=0.$$
Moreover,
  $$\sum_{k\in \Z_0^n} \frac{k_1^2}{|k|^{2\beta+2}}=\cdots =\sum_{k\in \Z_0^n} \frac{k_n^2}{|k|^{2\beta+2}}=\frac1n \sum_{k\in \Z_0^n} \frac{1}{|k|^{2\beta}}.$$
\end{lemma}

\begin{proof}
For any positive integer $\ell$, we define $\Lambda_\ell=\{k\in \Z_0^n: |k|^2=\ell\}$ which is a finite set (empty sets are considered to be finite). Then $\Z_0^n=\cup_{\ell=1}^\infty \Lambda_\ell$. To prove the first assertion, we assume without loss of generality that $i=1, j=2$. We have
  $$\sum_{k\in \Z_0^n} \frac{k_1k_2}{|k|^{2\beta+2}}=\sum_{\ell=1}^\infty \frac1{\ell^{\beta+1}}\sum_{k\in\Lambda_\ell} k_1k_2.$$
For any $k=(k_1,k_2,\cdots, k_n)\in\Lambda_\ell$, let $k^{(1)},k^{(2)}$ and $k^{(3)}$ be the three vectors in $\Lambda_\ell$ such that $k^{(i)}_j=k_j$ for all $j\in\{3,\cdots, n\}$ and $i=1,2,3$, and
  $$k^{(1)}_1=k_1,\quad k^{(1)}_2=-k_2; \quad k^{(2)}_1=-k_1,\quad k^{(2)}_2=k_2;\quad k^{(3)}_1=-k_1,\quad k^{(3)}_2=-k_2.$$
Let $k^{(0)}=k$. Then it is clear that $\sum_{i=0}^3 k^{(i)}_1 k^{(i)}_2=0$, which implies
  $$\sum_{k\in\Lambda_\ell} k_1k_2=0.$$
The first assertion is proved. The proof of the second one is similar. Indeed, if $k=(k_1,k_2,k_3,\cdots, k_n)\in \Lambda_\ell$, then $\bar k=(k_2,k_1,k_3,\cdots, k_n)\in \Lambda_\ell$, from which we conclude that
  $$\sum_{k\in \Lambda_\ell} k_1^2= \sum_{k\in \Lambda_\ell} k_2^2.$$
Thus
  $$\sum_{k\in \Z_0^n} \frac{k_1^2}{|k|^{2\beta+2}}=\sum_{\ell=1}^\infty \frac1{\ell^{\beta+1}} \sum_{k\in \Lambda_\ell} k_1^2 =\sum_{\ell=1}^\infty \frac1{\ell^{\beta+1}} \sum_{k\in \Lambda_\ell} k_2^2=\sum_{k\in \Z_0^n} \frac{k_2^2}{|k|^{2\beta+2}},$$
which finishes the proof.
\end{proof}

Therefore, by \eqref{torus-n.2} and Lemma \ref{sect-4-lem},
  $$\sum_{k\in \Z_0^n} \frac{\<u,k\>^2}{|k|^{2\beta+2}}=\sum_{i=1}^n u_i^2\sum_{k\in \Z_0^n} \frac{k_i^2}{|k|^{2\beta+2}} =\frac{|u|^2}n \sum_{k\in \Z_0^n} \frac{1}{|k|^{2\beta}}.$$
Combining this equality with \eqref{torus-n.1}, we arrive at
  \begin{equation*}
  \sum_{k\in \Z_0^n} \sum_{i=1}^{n-1}\big(\<A_{k,i}(\theta),u\>^2+\<B_{k,i}(\theta),u\>^2\big)= \frac{n-1}n |u|^2 \sum_{k\in \Z_0^n} \frac{1}{|k|^{2\beta}}=\nu_0 |u|^2,
  \end{equation*}
where
  $$\nu_0=\frac{n-1}n\sum_{k\in \Z_0^n} \frac{1}{|k|^{2\beta}}<+\infty.$$
Thus the system $\big\{\frac{A_{k,i}}{\sqrt{\nu_0}},\frac{B_{k,i}}{\sqrt{\nu_0}}:1\leq i\leq n-1, k\in \Z_0^n\big\}$ satisfies the condition (a).

Next,
  $$\nabla_{A_{k,i}} A_{k,i}=\frac{e_{k,i}}{|k|^\beta} \<A_{k,i}, \nabla_\theta \cos(k\cdot\theta)\>= -\frac{e_{k,i}}{|k|^{2\beta}} \cos(k\cdot\theta)\sin(k\cdot\theta)\<e_{k,i},k\>=0.$$
In the same way, $\nabla_{B_{k,i}} B_{k,i}=0$, hence the condition (b) is also verified. Finally, for any vector field $V$ on $\T^n$, we have
  $$\nabla_V A_{k,i}=\frac{e_{k,i}}{|k|^\beta} \<V, \nabla_\theta \cos(k\cdot\theta)\>=-\frac{e_{k,i}}{|k|^\beta} \sin(k\cdot\theta) \<V, k\>.$$
Similarly,
  $$\nabla_V B_{k,i}=\frac{e_{k,i}}{|k|^\beta} \cos(k\cdot\theta) \<V, k\>.$$
Then for $u_1,u_2 \in\R^n$,
  \begin{eqnarray*}
  && \<A_{k,i},u_1\> \<\nabla_V A_{k,i},u_2\> +\<B_{k,i},u_1\> \<\nabla_V B_{k,i},u_2\>\\
  &=& -\frac{\<V,k\>}{|k|^{2\beta}}\cos(k\cdot\theta) \sin(k\cdot\theta) \<e_{k,i},u_1\>  \<e_{k,i},u_2\> +\frac{\<V,k\>}{|k|^{2\beta}}\cos(k\cdot\theta) \sin(k\cdot\theta) \<e_{k,i},u_1\>  \<e_{k,i},u_2\>\\
  &=& 0.
  \end{eqnarray*}
Thus condition (c) is also satisfied.

Now let $\{u_t; \ t\geq 0\}$ be a family of $C^{2,\alpha}$-vector fields of divergence free on $\T^n$. Consider  the following SDE
  \begin{equation}\label{eq4.4}
  \begin{split}
  \d X_t &=\sqrt{\frac{2\nu}{\nu_0}} \sum_{k\in \Z^n_0}\sum_{i=1}^{n-1} \big( A_{k,i}(X_t)\circ \d W_t^{k,i} + B_{k,i}(X_t)\circ \d\tilde W_t^{k,i}\big) + u_t(X_t)\, \d t,\\
  X_0&=x\in\T^n,
  \end{split}
  \end{equation}
where $\big\{W_t^{k,i}, \tilde W_t^{k,i};\, 1\leq i\leq n-1, k\in \Z^n_0\big\}$ is a family of independent  standard real Brownian motions. When $\beta >2+n/2$, the SDE \eqref{eq4.4} defines a stochastic flow $\{X_t; \, t\geq 0\}$ of $C^1$-diffeomorphisms of $\T^n$ (see \cite{Cruzeiro} for the case $n=2$). In this case, by \eqref{A.39}, for almost surely $w$, $x\ra X_t(x,w)$ preserves the measure $\d x$; therefore by Theorem \ref{thm-manifold}, we have

\begin{theorem}\label{th4.1}
The velocity $u_t\in C^{2,\alpha}$ with initial value $u_0$ is a solution of the Navier--Stokes equations on $\T^n$ if and only if
  \begin{equation}\label{eq4.5}
  u_t= \E\Bigl[ {\mathbf P}\big( (X_t^{-1})^\ast u_0^\flat \big)^{\sharp}\Bigr].
  \end{equation}
\end{theorem}

\subsection{Case of sphere $\S^n$}

Let $\square$ be the de Rham--Hodge Laplacian operator acting on vector fields over $\S^n$. For $\ell\geq 1$, set $c_{\ell,\delta}=(\ell+1)(\ell+n-2)$. Then $\{c_{\ell,\delta};\ell\geq 1\}$ are the eigenvalues of $\square$ corresponding
to the divergence free eigenvector fields. Denote by ${\cal D}_\ell$ the eigenspace associated to $c_{\ell,\delta}$ and
$d_\ell=\hbox{\rm dim}({\cal D}_\ell)$ the dimension of ${\cal D}_\ell$. It is known that
  $$d_\ell\sim O(\ell^{n-1})\quad \hbox{\rm as}\ \ \ell\to +\infty.$$
For $\ell\geq1$, let $\{V_{\ell,k}; k=1, \ldots, d_\ell\}$ be an orthonormal basis of ${\cal D}_\ell$ in $L^2$:
  $$\int_{\S^n} \big\<V_{\ell,k}(x), V_{\alpha,\beta}(x) \big\>\, \d x =\delta_{\ell\alpha}\delta_{k\beta}.$$
Weyl's theorem implies that the vector fields $\{V_{\ell,k};k=1,\ldots,d_\ell,\, \ell\geq1\}$ are smooth. We refer to \cite{Raimond} for a detailed study on isotropic flows on $\S^n$, many properties below were proved there. But we are more familiar with \cite{FangZhang} to which we refer known results. Let $\{b_\ell;\ell\geq 1\}$ be a family of positive numbers such that $\sum_{\ell=1}^\infty b_\ell <+\infty$. Set
  $$A_{\ell,k}=\sqrt{\frac{n b_\ell}{d_\ell}}\, V_{\ell,k}.$$
Below we shall consider the family
  \begin{equation*}
  \bigl\{ A_{\ell,k};\ 1\leq k\leq d_\ell, \ell\geq 1\bigr\}.
  \end{equation*}

Let's first check the condition (a). By \cite[(A.13)]{FangZhang}, we have, for $x,y\in\S^n$
  \begin{equation}\label{eq4.7}
  \frac{n}{d_\ell}\sum_{k=1}^{d_\ell}\<V_{\ell,k}(x),y\>^2=\sin^2\theta,
  \end{equation}
where $\theta$ is the angle between $x$ and $y$. Let $u\in T_x\S^n$; then $\<x,u\>=0$. By \eqref{eq4.7},
  $$\frac{n}{d_\ell}\sum_{k=1}^{d_\ell}\<V_{\ell,k}(x),u\>^2=|u|^2.$$
Therefore,
  \begin{equation*}\label{eq4.8}
  \sum_{\ell\geq 1}\sum_{k=1}^{d_\ell}\<A_{\ell,k}(x),u\>^2= \sum_{\ell\geq 1} \frac{n b_\ell}{d_\ell}\sum_{k=1}^{d_\ell}\<V_{\ell,k}(x),u\>^2= \nu_0\, |u|^2,
  \end{equation*}
where
  \begin{equation*}\label{eq4.9}
  \nu_0=\sum_{\ell\geq 1} b_\ell.
  \end{equation*}

Next, by \cite[Propositions A.3 and A.5]{FangZhang},
  \begin{equation}\label{eq4.6}
  \sum_{k=1}^{d_\ell} \nabla_{V_{\ell,k}} V_{\ell,k}=0.
  \end{equation}
thus the condition (b) is satisfied.

It remains to check the condition (c). To this end, we need a bit more description on $V_{\ell,k}$. Let $\{e_1, \cdots, e_{n+1}\}$ be the canonical basis of $\R^{n+1}$. We denote by $P_0=e_{n+1}$ the north pole. When $n\geq 3$, the group $SO(n+1)$ acts transitively on $\S^n$. Let $x\in\S^n$ be fixed, then there is $g\in SO(n+1)$ such that $x=\chi_g(P_0)=gP_0$. Then
  \begin{equation}\label{eq4.9-2}
  V_{\ell,k}(gP_0)=\sqrt{\frac{d_\ell}{n}}\, \sum_{i=1}^n Q_{ki}^\ell(g)d\chi_g(P_0)e_i,
  \end{equation}
where $\{Q^\ell;\ \ell\geq 1\}$ is the family  of irreducible unitary representations of $SO(n+1)$ which keep the representation $h\ra d\chi_h(P_0)$. It is important that the element $Q_{qi}^\ell$ has an explicit formula for $1\leq q, i\leq n$:
  \begin{equation}\label{eq4.10}
  Q_{qi}^\ell(g)=\bigg(t\gamma_\ell(t)- \frac{1-t^2}{n-1}\gamma_\ell'(t)\bigg) g_{qi}-\bigg(\gamma_\ell(t)+\frac{t}{ n-1} \gamma_\ell'(t)\bigg)\, g_{q,n+1}g_{n+1,i},
  \end{equation}
with $t=g_{n+1,n+1}$ and
  \begin{equation*}
  \gamma_\ell(\cos\theta)=\int_0^\pi \bigl(\cos\theta-\sqrt{-1}\sin\theta\,\cos\varphi\bigr)^{\ell-1}\, \sin^n\varphi\, \frac{\d\varphi}{c_n},
  \end{equation*}
where $c_n=\int_0^\pi \sin^n\varphi\,\d\varphi$. Set $E_j=d\chi_g(P_0)e_j$; then $\{E_1, \cdots, E_n\}$ is an orthonormal basis of $T_x\S^n$. Fix $j$, we consider $\hat g(s)\in SO(n+1)$ which leaves invariant $e_i$ for $i\neq j$, $i\neq n+1$ and
  \begin{equation*}
  \left\{\begin{array}{ccc}\hat g(s)e_j&=&\cos s\, e_j-\sin s\,e_{n+1,}\\
				    \hat g(s) e_{n+1}&=&\sin s\,e_j+\cos s\, e_{n+1}.\end{array}\right.
  \end{equation*}
Then by \cite[p.596]{FangZhang},
  \begin{equation}\label{eq4.11}
  \nabla_{E_j}V_{\ell,k}(x)=\sqrt{\frac{d_\ell}{n}} \sum_{i=1}^n\sum_{\beta=1}^{d_\ell} Q^\ell_{k\beta}(g)\Bigl\{\frac{\d}{ \d s} \Big|_{s=0} Q^\ell_{\beta i}(\hat g(s))\Bigr\} E_i.
  \end{equation}
Combining \eqref{eq4.9-2} and \eqref{eq4.11}, we get
  \begin{align*}
  \sum_{k=1}^{d_\ell} V_{\ell,k}\wedge \nabla_{E_j}V_{\ell,k}&=\frac{d_\ell}{n} \sum_{q,i=1}^n\sum_{\beta, k=1}^{d_\ell} Q^\ell_{k\beta} Q^\ell_{kq}
  \Bigl\{\frac{\d}{ \d s} \Big|_{s=0} Q^\ell_{\beta i}(\hat g(s))\Bigr\} E_q\wedge E_i\\
  &=\frac{d_\ell}{n}  \sum_{q,i=1}^n \Bigl\{\frac{\d}{ \d s} \Big|_{s=0} Q^\ell_{qi}(\hat g(s))\Bigr\} E_q\wedge E_i.
  \end{align*}

In \eqref{eq4.10}, we replace $g$ by $\hat g(s)$; therefore $t=\cos s$, the term $g_{qi}=0$ for $q\neq i$, $g_{i,n+1}=0$ if $i\neq j$, $g_{n+1,i}=0$ if $i\neq j$. We have $g_{jj}=\cos s$ and $g_{n+1,j}g_{j,n+1}=-\sin^2s$. It follows that
  \begin{equation*}\label{eq4.13}
  \sum_{k=1}^{d_\ell} V_{\ell,k}\wedge \nabla_{E_j}V_{\ell,k}=0.
  \end{equation*}
The condition (c) is satisfied. Notice that using \eqref{eq4.10}, we have in fact the stronger result
  \begin{equation*}
  \sum_{k=1}^{d_\ell} V_{\ell,k}\otimes \nabla_{E_j}V_{\ell,k}=0.
  \end{equation*}

Now let $\{u_t; \, t\geq 0\}$ be a family of $C^{2,\alpha}$-vector fields of divergence free on $\S^n$. Let $b_\ell=1/ \ell^{1+\alpha}$. Consider the following SDE
  \begin{equation}\label{eq4.14}
  \d X_t =\sqrt{\frac{2\nu}{\nu_0}}\sum_{\ell\geq 1}\sum_{k=1}^{d_\ell}  A_{\ell, k}(X_t)\circ \d W_t^{\ell, k}  + u_t(X_t)\, \d t,\quad X_0=x\in\S^n,
  \end{equation}
where $\bigl\{W_t^{\ell, k};\ \ell\geq 1, 1\leq k\leq d_\ell\bigr\}$ is a family of independent  standard real Brownian motions.
When $\alpha >2 $, the SDE \eqref{eq4.14} defines a flow of $C^1$-diffeomorphisms of $\S^n$ (see \cite{LeJanRaimond, Luo}). In this case, for almost surely $w$, $x\ra X_t(x,w)$ preserves the measure $\d x$; therefore by Theorem \ref{thm-manifold}, we have

\begin{theorem}\label{th4.2}
The velocity $u_t\in C^{2,\alpha}$ with initial value $u_0$ is a solution of the Navier--Stokes equation on $\S^n$ if and only if
  \begin{equation}
  u_t= \E\Bigl[ {\mathbf P}\big( (X_t^{-1})^\ast u_0^\flat \big)^{\sharp}\Bigr].
  \end{equation}
\end{theorem}

\section{Appendix: gradient system on the sphere}\label{appendix}

For reader's convenience, we shall show that the gradient system  in the case of sphere $\S^n$ enjoy properties (a)--(c) in Section \ref{Extension}, but not (d).
 We denote by $\langle\, , \rangle$ the canonical inner product of $\R^{n+1}$. Let $x\in \S^n$, the tangent space $T_x\S^n$ of $\S^n$ at the point $x$ is given by
  \begin{equation*}\label{A1}
  T_x\S^n=\bigl\{v\in\R^{n+1};\ \langle v, x\rangle=0\bigr\}.
  \end{equation*}
Then the orthogonal projection $P_x: \R^{n+1}\ra T_x\S^n$ has the expression:
  \begin{equation*}\label{A2}
  P_x(y)=y-\langle x, y\rangle x.
  \end{equation*}

Let $\{e_1, \cdots, e_{n+1}\}$ be an orthonormal basis of $\R^{n+1}$; then the vector fields $A_i(x)=P_x(e_i)$ have the expression: $A_i(x)=e_i-\langle x, e_i\rangle\, x$ for $i=1, \cdots, n+1$. Let $v\in T_x\S^n$ such that $|v|=1$, consider
  \begin{equation*}\label{A3}
  \gamma(t)=x\,\cos t+v\, \sin t.
  \end{equation*}
Then $\{\gamma(t);\ t\in [0,1]\}$ is the geodesic on $\S^n$ such that $\gamma(0)=x, \gamma'(0)=v$. We have $A_i(\gamma(t))= e_i-\langle\gamma(t),e_i\rangle\,\gamma(t)$. Taking the derivative with respect to $t$ and at $t=0$, we get
  \begin{equation}\label{A4}
  (\nabla_vA_i)(x)=P_x\bigl(-\langle v,e_i\rangle x-\langle x, e_i\rangle v\bigr)=-\langle x, e_i\rangle v.
  \end{equation}
It follows that
  \begin{equation}\label{A5}
  \div(A_i)=-n\langle x, e_i\rangle.
  \end{equation}
Hence,
  \begin{equation}\label{A6}
  \sum_{i=1}^{n+1} \div(A_i)A_i=-n\sum_{i=1}^{n+1} \bigl(\langle x,e_i\rangle e_i-\langle x, e_i\rangle^2x \bigr) = -n (x-x)= 0.
  \end{equation}
Replacing $v$ by $A_i$ in \eqref{A4}, we have $\dis \nabla_{A_i}A_i=-\langle x, e_i\rangle e_i +\langle x,e_i\rangle^2 x$; therefore summing over $i$, we get
  \begin{equation}\label{A7}
  \sum_{i=1}^{n+1}\nabla_{A_i}A_i=0.
  \end{equation}

Now let $v\in T_x\S^n$ and $a,b\in T_x\S^n$, we have
  \begin{align*}
  \langle A_i\wedge \nabla_vA_i, a\wedge b\rangle&=\langle A_i,a\rangle\langle\nabla_vA_i,b\rangle-\langle A_i,b\rangle\langle\nabla_vA_i,a\rangle\\
  &=\langle a,e_i\rangle\langle x,e_i\rangle\langle v,b\rangle-\langle x,e_i\rangle\langle b,e_i\rangle\langle v,a\rangle.
  \end{align*}
Summing over $i$ yields
  \begin{equation}\label{A8}
  \sum_{i=1}^{n+1}\langle A_i\wedge \nabla_vA_i, a\wedge b\rangle=\langle a,x\rangle\langle v,b\rangle-\langle x,b\rangle\langle v,a\rangle=0.
  \end{equation}
Let $B$ be a vector field on $\S^n$; by \eqref{A4}, $\nabla_B A_i=-\<x, e_i\> B$. Using $\dis \L_{A_i}B=\nabla_{A_i}B-\nabla_B A_i$ and combining with \eqref{A5} and \eqref{A6}, we get that
  \begin{equation}\label{A9}
  \sum_{i=1}^{n+1} \div(A_i) \L_{A_i}B=- n B.
  \end{equation}

Finally we notice that by \eqref{A7}--\eqref{A9},  the vector fields $A_1, \cdots, A_{n+1}$ satisfy the conditions (a)--(c) but not (d) in Section \ref{Extension}.

\bigskip
\noindent\textbf{Acknowledgment.} The two authors would like to thank D. Elworthy for his interest on this work, and for drawing their attentions to Riemannian symmetric spaces. The second author is grateful to the financial supports of the National Natural Science Foundation of China (Nos. 11431014, 11571347), the Seven Main Directions (Y129161ZZ1) and the Special Talent Program of the Academy of Mathematics and Systems Science, Chinese Academy of Sciences.


\begin{thebibliography}{99}

\bibitem{Ambrosio} Ambrosio, L., Transport equation and Cauchy problem for BV vector fields. \emph{Invent. Math.} 158 (2004), 227--260.

\bibitem{AmbrosioF}  Ambrosio, L.; Figalli, A.,  Geodedics in the space of measure-preserving maps and plans. \emph{Arch. Rational Mech. Anal.} 194 (2009), 421--469.

\bibitem{AC1} Antoniouk, A.; Arnaudon, M.; Cruzeiro, A. B., Generalized stochastic flows and applications to incompressible viscous fluids. \emph{Bull. Sci. Math.} {138} (2014), no. 4, 565--584.

\bibitem{AC2} Arnaudon, M.; Cruzeiro, A. B., Lagrangian Navier--Stokes diffusions on manifolds: variational principle and stability. \emph{Bull. Sci. Math.} 136 (8) (2012), 857--881.

\bibitem{ACF} Arnaudon, M.; Cruzeiro, A. B.; Fang, Shizan, Generalized stochastic Lagrangian paths for the Navier--Stokes equation, arXiv:1509.03491v1.


\bibitem{Brenier}  Brenier, Y., The least action principle and the related concept of generalized flows for incompressible perfect fluids. \emph{J. Amer. Math. Soc.} 2 (1989), 225--255.

\bibitem{Brenier2} Brenier, Y., Minimal geodesics on groups of volume-preserving maps and generalized solutions of the Euler equations. \emph{ Comm. Pure Appl. Math.} 52 (1999), 411--452.

\bibitem{Flandoli} Busnello, B.; Flandoli, F.; Romito, M., A probabilistic representation for the vorticity of a three-dimensional viscous fluid and for general systems of parabolic equations. \emph{Proc. Edinb. Math. Soc.}  48 (2005),  295--336.

\bibitem{ChenWH} Chen, Weihuan; Li, Xingxiao, Introduction to Riemannian Geometry (in Chinese), volume 2, Peking University Press, 2004.

\bibitem{Chorin} Chorin, A. J., Numerical study of slightly visous flow. \emph{J. Fluid Mech.} 57 (1973), 785--796.

\bibitem{Cruzeiro} Cipriano, F.; Cruzeiro, A. B., Navier--Stokes equation and diffusions on the group of homeomorphisms of the torus. \emph{Comm. Math. Phys.} 275 (2007), no. 1, 255--269.


\bibitem{Constantin}  Constantin, P.; Iyer G., A stochastic Lagrangian representation of the three-dimensional incompressible Navier--Stokes equations. \emph{Comm. Pure Appl. Math.} {61} (2008), no. 3, 330--345.

\bibitem{DiPernaLions} Di Perna, R. J.; Lions, P. L., Ordinary differential equations, transport theory and Sobolev spaces. \emph{Invent. Math.} 98 (1989), 511--547.

\bibitem{Elworthy} Eells, J.; Elworthy, K.D, Stochastic dynamical systems, in ``Control Theory and Topics in Functional analysis'', Vol. III, 179--185, Intern. Atomic Energy Agency, Vienne, 1976.

\bibitem{ELL} Elworthy, K. D.; Le Jan, Y.; Li, Xue-Mei, On the geometry of diffusion operators and stochastic flows. \emph{Lecture Notes in Mathematics}, 1720, {Springer--Verlag}, 1999.

\bibitem{FangLuo} Fang, Shizan; Li, Huaiqian; Luo, Dejun, Heat semi-group and generalized flows on complete Riemannian manifolds. \emph{Bull. Sci. Math.} 135 (2011), 565--600.

\bibitem{FangLuoTh} Fang, Shizan; Luo, Dejun; Thalmaier, A., Stochastic differential equations with coefficients in Sobolev spaces. \emph{J. Funct. Anal.} 259 (2010), 1129--1168.

\bibitem{FangZhang} Fang, Shizan; Zhang, Tusheng; Isotropic stochastic flow of homeomorphisms on $S^d$ for the critical Sobolev exponent. \emph{J. Math. Pures Appl.} 85 (2006), no. 4, 580--597.

\bibitem{H} Helgason S., Differential Geometry and Symmetric spaces, Academic Press, 1962, New York and London.

\bibitem{KN} Kobayashi, S.; Nomizu, K., Foundations of differential geometry. Vol. II. Reprint of the 1969 original. Wiley Classics Library. A Wiley-Interscience Publication. \emph{John Wiley \& Sons, Inc., New York}, 1996.

\bibitem{Kunita} Kunita, H., Stochastic differential equations and stochastic flows of diffeomorphisms. \emph{\'Ecole d'\'et\'e de probabilit\'es de Saint-Flour}, XII--1982, 143--303, \emph{Lecture Notes in Math.}, 1097, Springer, Berlin, 1984.

\bibitem{Kunita2} Kunita, H., Stochastic flows and stochastic differential equations. Cambridge Studies in Advanced Mathematics, 24. \emph{Cambridge University Press, Cambridge}, 1990.

\bibitem{LeJanRaimond} Le Jan, Y.; Raimond, O., Integration of Brownian vector fields. \emph{Ann. Probab.} 30 (2002), 826--873.

\bibitem{LeJan} Le Jan, Y.; Sznitman A. S., Stochastic cascades and 3-dimensional Navier--Stokes equations. \emph{Probab. Theory Related Fields} 109 (1997), 343--366.


\bibitem{Luo} Luo, Dejun, Stochastic Lagrangian flows on the group of volume-preserving homeomorphisms of the spheres. \emph{Stochastics} 87 (2015), no. 4, 680--701.

\bibitem{Malliavin} Malliavin, P., Formule de la moyenne, calcul de perturbations et Th\'eor\`eme d'annulation pour les formes harmoniques. \emph{J. Funct. Anal.} 17 (1974), 274--291.

\bibitem{Pierfelice} Pierfelice, V., The incompressible Navier--Stokes equations on non-compact manifolds, arXive 1406.1644, 2014.

\bibitem{Raimond} Raimond, O., Flots browniens isotropes sur la sph\`ere. \emph{Ann. Inst. H. Poincaré Probab. Statist.} 35 (1999), 313--354.

\bibitem{Stroock} Stroock, D., An introduction to the analysis of paths on a Riemannian manifold. Mathematical Surveys and Monographs, 74. \emph{American Mathematical Society, Providence, RI}, 2000.

\bibitem{Teman} Temam, R., Navier--Stokes equations and nonlinear functional analysis. Second edition. CBMS-NSF Regional Conference Series in Applied Mathematics, 66. \emph{Society for Industrial and Applied Mathematics (SIAM), Philadelphia, PA}, 1995.

\bibitem{TemamW} Temam, R.; Wang, Shouhong, Inertial forms of Navier--Stokes equations on the sphere. \emph{J. Funct. Analysis} 117 (1993), 215--242.

\bibitem{Zhang} Zhang, Xicheng, A stochastic representation for backward incompressible Navier--Stokes equations. \emph{Probab. Theory Related Fields} 148 (2010), no. 1--2, 305--332.

\bibitem{Zhang2} Zhang, Xicheng,  Stochastic flows of SDEs with irregular coefficients and stochastic transport equations. \emph{Bull. Sci. Math.} 134  (2010),  no. 4, 340--378.

\bibitem{Zhang3} Zhang, Xicheng, Quasi-invariant stochastic flows of SDEs with non-smooth drifts on compact manifolds. \emph{Stochastic Process. Appl.} 121 (2011), 1373--1388.

\end{thebibliography}
\end{document}